\documentclass[a4paper,12pt]{amsart}

\usepackage{boxedminipage}
\usepackage[dvips]{graphics}
\usepackage{amsmath,amsfonts,amssymb,amsthm}
\usepackage{epsfig, fullpage}
\usepackage{xspace}
\usepackage[dvips]{color}
\usepackage{multicol}
\usepackage{wrapfig}
\usepackage[table]{xcolor}

\newtheorem{theorem}{Theorem}[section]
\newtheorem{proposition}[theorem]{Proposition}
\newtheorem{lemma}[theorem]{Lemma}
\theoremstyle{definition}

\newtheorem{example}[theorem]{Example}

\newcommand{\mat}{\textsc{Matlab}\xspace}
\newcommand{\supp}{\operatorname{supp}}

\def\N{\mathbb N}
\def\Z{\mathbb Z}
\def\R{\mathbb R}

\begin{document}
\title{Preconditioners based on Windowed Fourier Frames Applied to Elliptic Partial Differential Equations}
\author{Samir K. Bhowmik$^1$ and Christiaan C. Stolk$^2$}
\address{KdV Institute for Mathematics, University of Amsterdam, Amsterdam, Netherlands.
$1.$ S.K.Bhowmik@uva.nl,  $2.$ C.C.Stolk@uva.nl}

\thanks{This research was funded by the Netherlands Organisation for
Scientific Research through VIDI grant 639.032.509.}

\subjclass[2010]{Primary}
\date{\today}

\begin{abstract}
We investigate the application of windowed Fourier frames (WFFs) to
the numerical solution of partial differential equations, focussing on
elliptic equations.  The action of a partial differential operator
(PDO) on a windowed plane wave is close to a multiplication, where the
multiplication factor is given by the symbol of the PDO evaluated at
the wave number and central position of the windowed plane wave.  This
can be exploited in a preconditioning method for use in iterative
inversion. For domains with periodic boundary conditions we find that
the condition number with the preconditioning becomes bounded and the
iteration converges well. For problems with a Dirichlet boundary condition,
some large and small singular values remain. However the iterative inversion
still appears to converge well.
\end{abstract}
\maketitle

\section{Introduction}
Localization in the position-wave number space is an important concept
in partial differential and other operator equations. The large class
of pseudodifferential operators acts approximately local in the
position-wave number space.
Prominent examples where this is put to use in numerics are wavelets and
multigrid methods, methods that are closely related.
Multigrid and wavelet approaches are very successful
for elliptic problems, and have been applied also elsewhere. For wavelet
concepts we refer to \cite{Cohen2003}, the multigrid literature is very
large, see e.g.\ the book \cite{TrottenbergOosterleeSchueller2001}.

Wavelets provide a decomposition in scale and position. However,
they have little resolution in direction, i.e.\ the direction in the
wave-number space. More recently, a frame of functions
\cite{Mallat2009} called curvelets has been proposed
as a tool for numerical analysis of PDEs
\cite{CandesEtAl2006}. Curvelets provide additional localization
in direction. The localization in direction is better and better for the
smaller scales by a so called parabolic scaling: A fourfold smaller
scale leads to twice better resolution in direction and twice better
resolution in position. A method of solving a pseudodifferential
equation using curvelets, including a curvelet based approximation for
the operator inverse suitable for use as a preconditioner was
introducted in \cite{HerrmannMoghaddamStolk2008}.
Approximations
of pseudodifferential operators were discussed in \cite{Demanet:1113369}.

In this paper we analyze windowed Fourier frames, also referred to as
Gabor frames. They provide a decomposition of the position-wave number
space (``phase space'') into rectangular blocks. Compared to curvelets
it is a simpler decomposition, and easier to implement. Compared to
wavelets it still offers better directional resolution, although at
somewhat higher (log-linear) cost. Windowed Fourier frames are also
simpler than curvelets in the sense that they are generated by
translations and modulations of a single window function, while for
curvelets there is no such set of transformations that exactly maps
the one to the other.

Chan et. al \cite{RHCTFC01} has considered elliptic problems and
proposed circulant preconditioners to solve the resulting system of
equations using iterative techniques, for example, the conjugate
gradient method.  They prove that such preconditioners could be chosen
to reduce the condition number from $\mathcal{O}(n^2)$ to
$\mathcal{O}(n)$, where $n$ grid points have been chosen to discretize
the problem for a second order elliptic problem.
Some popular preconditioning techniques to solve linear systems using
iterative methods have been discussed in detail in \cite{KeC01},
\cite[Chap. 1]{Cohen2003}, \cite[Chap. 10]{JGVL001},
\cite[Chap. 7]{CDM001} and references there in. These techniques
includes use of positive definite matrices, incomplete $LU$ and
cholesky factorizations, multilevel, multigrid, wavelet
preconditioners and so on.  It has been noticed that condition numbers
are in control with most preconditioners and have slower growth
compared with the unpreconditioned system.

We believe that, while multigrid and wavelets are very important as
preconditioning methods, other possibilities should also be studied.
The particular phase space localization of wavelets is
well suited for elliptic equations, but not for other types of PDE,
for example the Helmholtz equation. Windowed Fourier frames, curvelets
or maybe other transforms are natural candidates to
study in more general settings than the elliptic problem.

Our purpose is therefore to introduce a preconditioning method based
on windowed Fourier frames, and to establish that it performs well for
certain discrete PDE's. We start with straightforward elliptic
problems, and include an example where the different phase space
localization properties provide an advantage for windowed Fourier
preconditioning.  For this study we consider a symmetric second order
elliptic BVP with a finite and a periodic domains. We first discretize
the PDE using a standard finite difference scheme. We use a
preconditioner based on windowed Fourier frames (WFFs) and the symbol
of the operator while solving the discrete PDE using
iterative linear solvers, to speed up the convergence.

The article is organized in the following way. We start by introducing
the preconditioners in Section~\ref{sec:precon_construct}. We study
boundedness and invertibility of a symmetrically preconditioned
operator in Section~\ref{f:boundednessofsolutions}.  Some numerical
test results are presented in Section~\ref{Section4:fff}. We finish this
study in Section~\ref{conclusions001} with some concluding remarks.

\section{Preconditioners based on the symbol of the operator and a
windowed Fourier frame}\label{sec:precon_construct}
In this section, we focus on introducing and defining the preconditioner
based on the symbol of the partial differential operator and a windowed
Fourier frame. As a model problem we consider the boundary value problem
\begin{equation}\label{f:ellipticbvp01_a}
  \mathcal{L} u = f \quad \text{in}\quad \Omega ,
\end{equation}
with Dirichlet and periodic boundary conditions on $\partial\Omega$ where
\[
  \mathcal{L} u =-\nabla\cdot (a(x) \nabla u)+ b(x) u(x) ,
\]
for given real coefficients $a(x)$, and $b(x)$.
Here we consider $\Omega
\subset \mathbb{R}^d$, an open bounded domain,
$f:\Omega\longrightarrow \mathbb{R}$ is a given function and
$u:\Omega\cup \partial \Omega \longrightarrow \mathbb{R}$ is an
unknown function. A detailed description of this type of problems can
be found in \cite{NHA001, LCE, YPJR}, and many references therein.
We investigate both periodic and non-periodic boundary value problems.

We make the following assumptions on the coefficients, to ensure
that $\mathcal{L}$ is boundedly invertible. We assume there is some
constant $C$ such that $a(x) \ge C$. In the case of Dirichlet boundary
conditions we consider $b \ge 0$, in the case of periodic boundary
conditions we assume that $b \ge C_0 > 0$ for some constant $C_0$.

\subsection*{Windowed Fourier frames}
We first give a short introduction of windowed Fourier frames (WFF).
Let $\mathbf{H}$  be a Hilbert space.
A sequence $\{\psi_n\}_{n\in\Gamma}$ is a frame~\cite[Section 5.1.1, Definition 5.1]{Mallat2009} of \textbf{H} if there exist two constants $A>0$, $B>0$ such that for any $f \in \mathbf{H}$,
\begin{equation}\label{f:fram_condition001a}
 A\|f\|^2 \le \sum_{n\in\Gamma}|\langle f, \psi_n \rangle|^2 \le B\|f\|^2.
\end{equation}
The index set $\Gamma$ might be finite or infinite and one can define a frame operator $F_1$ so that
\[
  F_1f[n] = \langle f, \psi_n   \rangle , \qquad \text{for all} \quad n\in\Gamma.
\]
If the condition  (\ref{f:fram_condition001a}) is satisfied then $F_1$ is called a frame operator. When $A=B$ the frame is said to be tight~\cite[page 155]{Mallat2009}, \cite{OCH}.

A window function is simply a function in $C_0^\infty(\R^d)$, i.e.\
it is smooth, and zero outside some chosen finite domain.
Let us focus on $d=1$, and consider a real symmetric ($ g(t)=g(-t)$,
for all $t\in \mathbb{R}$), non-negative
window function.  We can translate $g$ by $v\in \R$ and modulate $g$ by frequency $\xi\in \R$ as
$
 g_{v, \xi}(t)= e^{i\xi t} g(t-v),
$
which are known as windowed Fourier atoms or Gabor atoms.
Here $\|g_{v,\xi}(t)\| = 1$ for any $v\in \R$, $\xi \in \R$ since $\|g\|=1$.
If the functions $g_{v, \xi}(x)$ satisfy the frame condition (\ref{f:fram_condition001a}), then they are called windowed Fourier frames (WFF)~\cite[Section 5.4]{Mallat2009} and
for any
$f \in \mathbf{L}^2(\mathbb{R})$, the operator $F$ 
  defined by
\begin{equation}\label{f:wff_eqn001}
 Ff(v, \xi) = \langle f, g_{v, \xi}\rangle = \int_{\R} f(t) g(t-v) e^{-i\xi t} dt,
\end{equation}
is called the windowed Fourier frame operator. In
Figure~\ref{windowfunctions001:f} we present a sample windowed
Fourier frame.
\begin{figure}[ht]
\begin{center}
\includegraphics[height=4cm]%
  {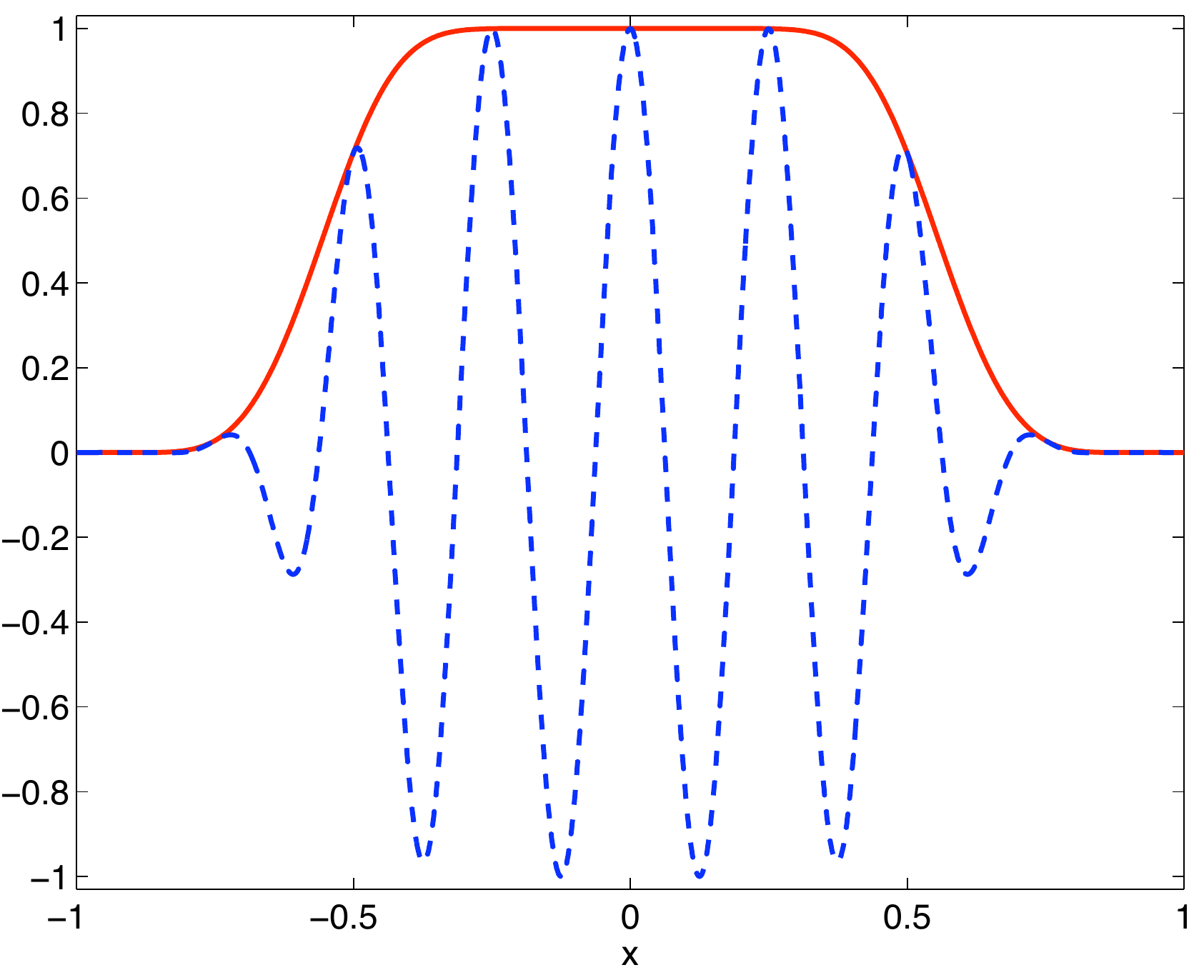}
\end{center}
\caption{%
Example of a function from a windowed Fourier frame in one dimension
with window function
$w(x)=\sin\left( \frac{\pi}{2}\frac{e^{-c/x}}{e^{-c/x}+e^{-c/(1-x)}} \right)$, $c=1.5$.%
%
}
\label{windowfunctions001:f}
\end{figure}

A version of the windowed Fourier frame transformation with discrete
index set $\Gamma$, is obtained by restricting $(v, \xi)$ to a rectangular
grid with interval size $v_0$ and $\xi_0$ in time and frequency
respectively and define~\cite[page 182]{Mallat2009}
\[
 g_{n, k} = g(t-nv_0) e^{i  \xi_k t},
\]
where $\xi_k = k \xi_0$, which will be needed in
Section~\ref{Section4:fff} while implementing the preconditioners
for elliptic PDE's. It is well understood \cite{ID01,Mallat2009}
that the windowed Fourier family $\{g_{n, k}\}_{(n, k) \in \Z^2}$ is a
frame only if $\frac{2 \pi}{v_0 \xi_0} \ge 1$ and the frame bounds $A$
and $B$ satisfies $A\le \frac{2 \pi}{v_0 \xi_0} \le B$.
Let $g$ be a window function with support
$[-\frac{\pi}{\xi_0}, \frac{\pi}{\xi_0}]$.
Then $\left\{g_{n, k}\right\}_{(n, k)\in \Z^2}$ is a tight frame~\cite{Mallat2009} with a frame bound equal to $A$, if
\begin{equation}\label{findbound0a0:f}
 \frac{2\pi}{\xi_0}\sum_{n=-\infty}^{\infty} |g(t-nv_0)|^2 = A >0
\end{equation}
for all $t\in \R$.

In the discrete setting, we replace the Fourier basis
$\{e^{i k \xi_0  t}\}_{k\in\Z}$ of $L^2[-\pi/\xi_0, \pi/\xi_0]$ by the discrete
Fourier basis $\{e^{\frac{i 2 \pi k n}{N_l}}\}_{0 \le k< N_l}$ of
$\mathbb{C}^{N_l}$ to construct a discrete windowed Fourier frame.
Let us consider $g[n]$ be a $N$ periodic discrete window function with
a support and restricted to $[-N/2, N/2]$ that is included in
$[-N_l/2, N_l/2-1]$.  Then following ~\cite{Mallat2009} one can see
that if $\mod(N, M)=0$ and
\[
 N_l \sum_{m=0}^{\frac{N}{M}-1} |g[n-mN]|^2 =A>0 \quad \forall \quad 0\le n<N
\]
then $\{g_{m, k}[n]= g[n-mM]e^{i 2\pi kn/N_l}\}_{0\le k <N_l, 0\le m <K}$ is a tight frame in $\mathbb{C}^N$ with frame bound equal to $A$ where
$K=\frac{N}{M}$.
For a fixed window position indexed by $m$, the discrete windowed Fourier frame coefficients are
\[
Ff[m, k] = \langle f, g_{m, k}\rangle = \sum_{n=0}^{N-1} f[n]g[n-mM]e^{-\frac{i2\pi kn}{N_l}},
\]
for all $0\le k< N_l$.

We set the length of the window equal to $2 v_0$, so that two successive
windows have overlapping support and we choose tight windowed Fourier frames.
So we choose the parameters such that
$\frac{2\pi}{\xi_0} = 2 v_0$, and $g$ has support in $[-v_0,v_0]$.

Several choices for the window  function are discussed in
Appendix~\ref{f:appendix}. Our initial criterion for a good window function
is the decay of its Fourier transform.
We conclude that both the windows $h_5(x)$ 
with $c=1.5$, $d=0.9$ and the
window formed by $h_4(x)$ are very well behaved. We
experimented with both the windows  to define the preconditioners and recorded
the number of iterations needed for convergence while solving the
precondition linear system. The numbers for these two choices of the window
function are approximately equal, preconditioners formed by using the
stretched window $h_5(x)$ with  $c=1.5$, $d=0.9$  take a few iterations
less than that of $h_4(x)$, but the effect is small. We decide to use the
window $h_5(x)$ for this study since we designed it
(and it performs a little better).

In order to apply windowed Fourier frames to the BVP
\eqref{f:ellipticbvp01_a} we need to define and organise window
functions for a multi-dimensional and bounded domain $\Omega \subset \R^d$,
say $\bar{\Omega}=[a, b]^d$, where the domain can be both periodic and
non-periodic. Multidimensional window functions are defined
straightforwardly using a tensor product approach. For periodic
domains we assume that the period is a discrete multiple of $v_0$, say $Kv_0$.
Then the set of coefficients becomes periodic with period $K$.
In Section~\ref{f:boundednessofsolutions} we will discuss this further.
For non-periodic problems, we define
$K+1$ windows  in each of the coordinate directions (instead of $K$ windows) with support $\frac{4\pi}{K}$ (in each of the coordinate directions), considering  an extended domain $[a-\frac{2\pi}{K}, b+\frac{2\pi}{K}]^d \supset \Omega$,
by
\[
 g_j(x) =\left\{
\begin{array}{cc}
 g(x-jv_0) & x \in \Omega,\\
 0      & x\notin\Omega,
\end{array}
\right.
\]
where $j \in \{1, 2, \cdots, (K+1)^d\}$. Then we have $2^d(K+1)$,
($d>1$) windows that cross the boundary of $\Omega$.

\subsection*{Preconditioners}
Here we return to our main discussion. First we give a short motivation why
a PDO can be approximated by WFF's and the symbol of the PDO.
For simplicity we start with BVP (\ref{f:ellipticbvp01_a}) in one dimension.
Setting
\[
  u = g_{j, k}(x) = e^{i\xi_k x} g(x-jv_0), \quad \text{$x\in \Omega\subset \R$},
\]
$j\in J$ (the set $J$ is defined in Section~\ref{f:boundednessofsolutions}), $k\in \Z$,
we get
\begin{align*}
  a(x) \frac{\partial^2 u}{\partial x^2}
  = {}& a(x)\frac{\partial}{\partial x} \left[i \xi_k e^{i \xi_k x}
    g(x-jv_0)+e^{i \xi_k x} \frac{d }{dx}g(x-jv_0)\right]
\\
  = {}& a(x)\left(-\xi_k^2 e^{i \xi_k x} g(x-jv_0)+ (1+i \xi_k) e^{i \xi_k x}
    \frac{d }{dx}g(x-jv_0)+ e^{i \xi_k x} \frac{d^2}{dx^2} g(x-jv_0)\right)
\\
  = {}& A_1\qquad+\qquad B\qquad+\qquad C,
\end{align*}
where $A_1$ is the leading part (when $\xi_k$ is very large) of the
elliptic operator acting on (\ref{f:ellipticbvp01_a}).  When $\xi_k$
is very large, $\xi_k^2$ is the dominating term in $a(x)\frac{d^2}{dx^2}$
and so the differential operator $a(x)\frac{d^2}{dx^2}$ can be approximated
by a multiplication operator $a(x) \xi_k^2$.
Now let us represent a function $u(x)\in \mathbf{H}$ using windowed
Fourier frames
\[
 u(x)=\sum_{j, k} C_{j, k} g_{j, k} (x),\quad \text{$x\in \Omega$}
\]
where
\[
 C_{j, k} = \langle u(x), g_{j, k}(x) \rangle
\quad \text{and}\quad
g_{j, k}(x) = e^{i \xi_k x} g(x-jv_0).
\]
Then we can approximate
\[
a(x)\frac{\partial^2 u}{\partial x^2} \approx \sum a(x) \xi_k^2 C_{j, k} g_{j, k}(x)
= \sum a(x) \xi_k^2 \langle u, g_{j, k}\rangle g_{j, k}.
\]
Here we observe that the PDO ($a(x)\frac{\partial^2 u}{\partial x^2}$)
can be approximately reconstructed by the tight windowed Fourier
frames and its duals multiplied by the symbol of the PDO if $a$ varies
little over the support of $g_{j,k}$. This result motivates us to define
preconditioners based on symbols of the PDO's and WFF's.

It is to note that we consider $\bar x_j$, $j\in J$, as midpoints of the subdomains of $\Omega$.
To bound the condition number and to speed up the convergence of
iterative solvers we are interested in defining preconditioners for
a discrete equivalent of (\ref{f:ellipticbvp01_a}) based on the symbol
and WFFs.  To this end, we define an invertible matrix $M$
\[
  M_{\tilde{j},\tilde{k};j,k}
  = \delta_{\tilde{j},j} \delta_{\tilde{k},k} (1+a(\bar x_j) \xi_k^2 ).
\]
We solve an equivalent system (of $\mathcal{A}u=f$, in which $\mathcal {A}$
is a symmetric discrete equivalent of the operator acting on
(\ref{f:ellipticbvp01_a}) using a finite difference scheme).
Several forms of preconditioning can be distinguished:
\begin{description}
 \item[Symmetric Preconditioning]
$
 P \mathcal{A} P \tilde u = P f$,
coupled with
 $ P \tilde u = u$,
where $P=F^* M^{-\frac{1}{2}} F$.
\item[Left Preconditioning]
$
 P \mathcal{A}  u = P f$,
where $P=F^* M^{-1} F$.
\item[Right Preconditioning]
$
 \mathcal{A} P \tilde u = f$, coupled with $u=P \tilde u$
where $P=F^* M^{-1} F$.
\end{description}
In each case $F^*$ is the conjugate transpose of the frame operator $F$.
The relation between the left and right preconditioners is established
in the following result.
\begin{theorem}
The singular values of the left preconditioned matrix $P\mathcal{A}$ and
the  right preconditioned matrix $\mathcal{A}P$ are equal where
$P=F^* M^{-1} F$.
\end{theorem}
\begin{proof}
 We have
\[
 (P\mathcal{A})^*=\mathcal{A}^* P^*= \mathcal{A}P,
\]
since $\mathcal{A}$ is symmetric and
\[
 P^*=(F^* M^{-1} F)^*=(M^{-1} F)^*F=F^* M^{-1} F,
\]
proves the claim since singular values of $P\mathcal{A}$ and
$(P\mathcal{A})^*$ are the same.
\end{proof}

%

\section{Boundedness and invertibility of the symmetrically preconditioned
operator $P\mathcal{L}P$}\label{f:boundednessofsolutions}
A key property of a preconditioner, is that the condition number of
the preconditioned operator is bounded independent of the
discretisation. We will show that the symmetrically preconditioned
operator $P\mathcal{L}P$, where $P=F^* M^{-1/2} F$ and $F$ stands for
tight windowed Fourier frame operator, considered in the continuous
case, defines a continuous map, with continuous inverse on $L^2(\Omega)$.
Under a suitable discretisation this can be used to derive the first
mentioned property, although we will not do so (instead we study some
discrete systems numerically in the next section). We prove this in
the case of the domain
\[
 \Omega = [0, 2\pi)^n = \mathbb{T}^n
\]
i.e.\ $[0, 2\pi]^n$  with the periodic boundary conditions. In the case
of a domain with a boundary we will see in the section on numerics that there
are a few singular values that grow if the discretisation parameter becomes
small, so that the result appears to be false.

\begin{theorem}\label{f:theorem001}
The self adjoint operator $P\mathcal{L}P:L_2(\Omega) \rightarrow L_2(\Omega)$
is boundedly  invertible, that is, there exists $c_2 \ge c_1>0$ such that
\[
  c_1 \|u\|^2 \le \langle P\mathcal{L}P u, u \rangle \le c_2 \|u\|^2,
    \quad \text{for all}\quad u\in H.
\]
\end{theorem}

By assumption $\mathcal{L}: H^{1}\rightarrow H^{-1}$ is boundedly invertible.
Therefore, to establish Theorem~\ref{f:theorem001}, it is sufficient to show
that $P$ is boundedly invertible from $L_2(\Omega)$ to $H^1(\Omega)$,
and that $P$ is boundedly invertible from $H^{-1}(\Omega)$ to $L_2(\Omega)$.
In fact it sufficient to show that $P$ is boundedly invertible from
$L_2(\Omega)$ to $H^1(\Omega)$, as this implies that $P^*$ is boundedly
invertible from $H^{-1}(\Omega)$ to $L_2(\Omega)$ and $P = P^*$.

One may think of $\mathbb{T}^n$ as the hypercube $[0, 2\pi[^n
\subset \R^n$ or $[-\pi, \pi [^n$. Functions on $\mathbb{T}^n$
may be thought as those functions on $\R^n$ that are $2 \pi$
periodic in each of the coordinate directions.
Let $p \in \R$. The sobolev space $H^p(\Omega)$  is the space of all
functions $\psi \in L^2(\Omega)$  that satisfy
\begin{equation}\label{f:defn002a}
 \|\psi\|_p^2=\sum_{m\in \Z^n} (1+|m|^2)^p |a_m|^2< \infty
\end{equation}
for the Fourier coefficients $a_m$ of $\psi$. The space $H^p(\Omega)$ is
a Hilbert space with the scalar product defined by
\begin{equation}\label{f:defn002a_a}
 \langle \phi, \psi \rangle_p := \sum_{m\in \Z^n} (1+|m|^2)^p a_m \bar b_m
\end{equation}
for $\phi$, $\psi \in H^p(\Omega)$ with Fourier coefficients $a_m$ and
$b_m$, respectively.

When we consider one window only, the windowed Fourier frame operator
$F$ becomes the Fourier transform operator. So we first prove the
invertibility of the operator $P\mathcal{L}P$ considering $F$ as the
Fourier transform operator.
From here we denote  $M_{\frac{1}{2}} =M^{-\frac{1}{2}}$. We first show that
$Pu\in {H^1(\Omega)}$ for $u\in L_2(\Omega)$. Here
\[
  \langle Pu, Pu\rangle
  = \left\langle F^{-1} M_{\frac{1}{2}} Fu, F^{-1}M_{\frac{1}{2}} Fu\right\rangle
  = \left\langle M_{\frac{1}{2}} Fu,M_{\frac{1}{2}} Fu\right\rangle
\]
where $Fu\in \ell_2(\Z)$ is the sequence of Fourier coefficients of
$u\in L_2(\Omega)$.
Now using (\ref{f:defn002a_a}) we have
\begin{eqnarray*}
\langle Pu, Pu\rangle_{1} & = & \sum_{m\in \Z^n} (1+|m|^2)
(M_{\frac{1}{2}})^2 |(Fu)_m|^2,
\end{eqnarray*}
where $(Fu)_m\in \ell_2(\Z)$ are the Fourier coefficients of
$u\in L_2(\Omega)$.
Now for any $m\in\Z^n$ there exists $0<\tilde A< \tilde B$ such that
$
 \tilde A \le (1+|m|^2)(M_{\frac{1}{2}})^2  \le \tilde B.
$
So  one gets
$$
 \tilde A \|u\|^2 \le \|Pu\|_1 \le \tilde B \|u\|^2,
$$
and
$$ Pu = F^{-1}M_{\frac{1}{2}}Fu \in H^1(\Omega),\quad \text{ {if}
$u\in L_2(\Omega)$}.
$$
Thus there exists $0<c_1 \le c_2$ such that
\[
 c_1\|u\|^2 \le \langle P\mathcal{L}Pu, u \rangle \le c_2 \|u\|^2,
\]
and that confirms the boundedly invertibility of
$P \mathcal{L} P : L_2(\Omega) \rightarrow L_2(\Omega)$, since
$\mathcal{L}$ is boundedly invertible and $P^* = P$.

Next we investigate the boundedly invertibility of
$P \mathcal{L} P : L_2(\Omega) \rightarrow L_2(\Omega)$ considering $F$
as a windowed Fourier frame operator. To define window functions we
consider $K$ uniform subintervals of size $\frac{2\pi}{K}$ in each of
the coordinate directions so that we can divide $\Omega$
($=\mathbb{T}^n$) into $K^n$ subdomains. Then we define $K^n$
overlapping subdomains with length $\frac{4\pi}{K}$ on each of the
coordinate directions and denote them by $D_j\subset\Omega$ for all
$j\in J$ where $J=\{1, 2, \cdots, \tilde K\}$ with $\tilde K=K^n$. We
consider $K$ windows on each of the coordinate directions.  We define
window functions $g_j(x)$ on $\Omega$ for all $j\in J$ satisfying
$\sum_{j\in J} g^2_j(x)=1$; where $x\in\Omega$.  The support of each $g_j(x)$
is contained in $D_j$. By $x_j$ we denote the midpoint of the block
$D_j$. Of course for the periodic setting one needs to
consider appropriate boundary domains, in one dimension for example,
$D_{1}=2\pi \left[0, \frac{1}{K}\right]\cup 2\pi \left[\frac{K-1}{K},
1\right)$. We will also assume that the support of $g_j$ stays
away some small distance from the boundaries.  The windowing
operator is denoted by $W$, defined by
\begin{align}
  W : L^2(\Omega) \longrightarrow {}&
    L^2(D_1)\times L^2(D_2)\times\cdots\times L^2(D_{\tilde{K}})
\\ \label{f:win_funct01}
 Wu= {}& \left(g_1 u, g_2 u, \cdots, g_{\tilde{K}} u \right).
\end{align}
($Wu$ can also be viewed as an element of the larger space
$L^2(\Omega, C^{\tilde{K}})$.)
The adjoint windowing operator $W^*$ then equals
\[
  W^*(u_1, u_2, u_3, \cdots, u_{\tilde{K}})= \sum_{j\in J} g_j u_j.
\]
The Fourier transform of the windowed Fourier transform is done on each
domain $D_j$. We will write $\mathcal{F}_{D_j}$ for this Fourier transform
on functions restricted to $D_j$.  We redefine the operator $P$ as
\[
  P = \sum_{j\in J} g_j \left(
        \mathcal{F}_{D_j}^{-1} M_{\frac{1}{2}} \mathcal{F}_{D_j}\right)g_j .
\]

We first study $P$ as an operator $H^{1/2} \rightarrow H^{-1/2}$, starting
with a result about $W$.
\begin{lemma}\label{f:lemma_inv_p}
Let $W$ be defined by (\ref{f:win_funct01}). Then there exists
$0 < \alpha < \beta$ such that
\begin{equation}\label{f:eq001a}
  \alpha \|Wu\|_{H^{1/2}(\Omega, C^{\tilde{K}})}
    \le \|u\|_{H^{1/2}(\Omega, C)}
    \le \beta \|Wu\|_{H^{1/2}(\Omega, C^{\tilde{K}})} .
\end{equation}
\end{lemma}

\begin{proof}
We consider separately the two inequalities
\begin{equation}\label{f:eq001b}
\alpha \|W u\|_{H^{1/2}(\Omega,C^{\tilde{K}})} \le \|u\|_{H^{1/2}(\Omega,C)}
\end{equation}
and
\begin{equation}\label{f:eq001c}
\|u\|_{H^{1/2}(\Omega,C)} \le \beta \|W u\|_{H^{1/2}(\Omega,C^{\tilde{K}})}.
\end{equation}
Equation (\ref{f:eq001b}) is easy, it follows directly from the
continuity of the multiplication by $g_j$ on $H^{1/2}$.
Equation (\ref{f:eq001c}) is more difficult. We solve it as follows.
Define $E_{(s)}$ to be the pseudodifferential on $\Omega$ ($=\mathbb{T}^n$,
the torus) with symbol $(1 + \|\xi\|^2)^{s/2}$.
This defines a self adjoint operator.
The $H^s$ norm of a function $u$ is
equivalent to $ \|E_{(s)} u\|_{L_2}.$
In this proof we in fact use this formula for the norm.

We start with the basic estimate
\begin{equation}\label{f:eq001d}
 \|Wu\|^2_{H^{1/2}} \ge \frac{1}{2} \|Wu\|^2_{L_2} +\frac{1}{C} \|E_{(1/2)}Wu\|^2
\end{equation}
where $C \ge 2$ is a constant to be chosen later. For $\|Wu\|^2_{L_2}$ we
find the following
\[
  \|u\|^2_{L^2}
  = \langle u, \sum_j g_j^2 u  \rangle
  = \sum_j  \langle g_j u, g_j u  \rangle
  = \|Wu\|_{L_2}^2.
\]
For
$
 \|E_{(1/2)}Wu\|^2
$
we use the following
\begin{align*}
\|E_{(1/2)} Wu\|^2_{L_2}
  = {}& \sum_j \langle E_{(1/2)} g_j u, E_{(1/2)} g_j u\rangle
\\
  = {}& \langle E_{(1/2)} u, E_{(1/2)} \sum_j g_j^2  u\rangle
    + \sum_j \langle u,[ E_{(1)}, g_j] g_j u\rangle,
\\
  = {}& \| u \|^2_{H^{1/2}}
    + \sum_j \langle u,[ E_{(1)}, g_j] g_j u\rangle,
\end{align*}
where we used that  $E_{(1/2)} E_{(1/2)}=E_{(1)}$ and
$ [E_{(1)}, g_j]=  E_{(1)} g_j - g_j E_{(1)}.$
The operator $[E_{(1)}, g_j]$ is a pseudodifferential operator of order
$0$ (follows from \cite{SAPG}).  Therefore we find that there is a
constant $D$ such that
\[
\left|\sum_j \langle u,[ E_{(1)}, g_j] g_j u\rangle\right| \le D \|u\|_{L_2}^2.
\]
We now use what we just described in (\ref{f:eq001d}) and we choose $C=\min(2D, 2)$. This yields
\begin{eqnarray*}
\|Wu\|^2_{H^{1/2}} &\ge& \frac{1}{2} \|u\|^2_{L_2}  +\frac{1}{C} \left(\|E_{(1/2)}u\|^2_{L_2}-D\|u\|_{L_2}^2 \right)\\
& \ge & \frac{1}{C} \|E_{(1/2)}u\|^2_{L_2} = \frac{1}{C} \|u\|^2_{H^{1/2}}.
\end{eqnarray*}
We therefore have proved (\ref{f:eq001c}).
\end{proof}

Using the lemma we can prove the following
\begin{proposition} \label{prop:P_Hhalf_inv}
$P$ is boundedly invertible $H^{-1/2}(\Omega) \rightarrow H^{1/2}(\Omega)$.
\end{proposition}
\begin{proof}
The operator $W$ maps from $H^{1/2}(\Omega)$ to  the Sobolev space
$H^{1/2} {(D_1 )}\times H^{1/2} {(D_2)}\times \cdots \times H^{1/2} (D_{\tilde{K}})$.
The proposition follows from the lemma we just proved and
the Lax-Milgram theorem, using the Fourier coefficient
description of this Sobolev spaces.
\end{proof}

\medskip

Next we study $P$ as an operator $L^2(\Omega) \rightarrow H^{-1}(\Omega)$,
by first proving that $P$ is a pseudo-differential operator.
\begin{proposition}\label{f:PsDOorder-1}
If the window functions $g_j(x) \in C_0^{\infty}(D_j)$,
then  the operator $P$ defined by
\[
  P= \sum_{j\in J} P_j=\sum_{j\in J} g_j S_j g_j
\]
is a periodic elliptic PsDO of order $-1$ where
$
 S_j = \mathcal{F}_{D_j}^{-1} M_{1/2} \mathcal{F}_{D_j}
$
is a periodic elliptic PsDO of order $-1$.
\end{proposition}
Before we go to prove the proposition, we will need some definitions and
results about periodic pseudodifferential operators
\cite{M.Ruzhansky.V.Turunen2006, V.TurunenG.Vainikko1998}.
Let $\mathcal{D}(\mathbb{T}^n)$ be the vector space
$C^{\infty}(\mathbb{T}^n)$ endowed with the usual test function topology.
Then any continuous linear operator $A:
\mathcal{D}(\mathbb{T}^n)\rightarrow \mathcal{D}(\mathbb{T}^n)$ can be
represented as
\begin{equation}\label{f:smbl001a}
  (Au)(x) = \sum_{\xi\in \Z^n} \sigma_A (x, \xi) \hat u(\xi) e^{i x\cdot \xi},
\end{equation}
where
\begin{equation}\label{f:smbl001b}
  \sigma_A(x, \xi) = e^{-i x\cdot \xi} A e^{i x\cdot \xi}
\end{equation}
is called the symbol of the operator $A$.
Let $m\in \R$ and $0 \le \delta < \rho \le 1$. In our case $\delta=0$ and
$\rho=1$.
An operator defined by (\ref{f:smbl001a}) is called a periodic
pseudo-differential operator (PsDO) of order $\alpha$
if the unique function $\sigma_A \in C^{\infty}(\mathbb{T}^n \times \Z^n)$
defined by (\ref{f:smbl001b})
%
satisfies
\begin{equation}\label{f:smbl001c}
|\Delta_\xi^{\alpha} \partial_x^{\beta}\sigma_A (x, \xi)| \le C_{\sigma \alpha \beta m}
\langle \xi \rangle^{m- \rho|\alpha|+ \delta |\beta|}
\end{equation}
for every $x\in \mathbb{T}^n$, for every $\alpha, \beta \in \N^n$ and
$\langle \xi \rangle := (1+|\xi|^2)^{1/2}$.
Here by $S^m_{\rho, \delta}(\mathbb{T}^n \times \Z^n)$ we denote the space of
functions $\sigma_A \in C^\infty(\mathbb{T}^n \times \Z^n)$ that satisfies (\ref{f:smbl001c}).
If $\sigma_A \in S^m_{\rho, \delta}(\mathbb{T}^n \times \Z^n)$, one may denote $A \in \operatorname{Op}S^m_{\rho, \delta}(\mathbb{T}^n \times \Z^n)$.
An operator $A \in \operatorname{Op}S^m_{\rho, \delta}(\mathbb{T}^n \times \Z^n)$
is called elliptic if, in addition,  the symbol
$\sigma_A \in S^m_{\rho, \delta}(\mathbb{T}^n \times \Z^n)$ of $A$ satisfies
 $|\sigma_A(x, \xi)| \ge C |\xi|^{m}$
where $x \in\R^n$, $C>0$ and $|\xi| \ge \xi_0 \ge 0$.

\begin{proof}[Proof of proposition~\ref{f:PsDOorder-1}]
First we  show that the operator
$S_j=\mathcal{F}_{D_j}^{-1} M_{1/2} \mathcal{F}_{D_j}$
is a periodic elliptic PsDO with period $\frac{4\pi}{K}$ on each of the coordinate directions.
Now for any $u \in \mathcal{D}(\mathbb{T}^n)$ we have
$ S_j u  =  \mathcal{F}_{D_j}^{-1} \left(M_{1/2} \left(\mathcal{F}_{D_j} u(\xi)\right)\right)$,
with  symbol
$a_j(\xi)=M_{\frac{1}{2}}(\xi) = \frac{1}{\sqrt{1+\xi^2}}$.
Consider $v_k\in\Z^n$ with $(v_k)_k=1$ and $(v_k)_i=0$ if $i\ne k$.
Now using difference calculus
\begin{eqnarray*}
\Delta_{\xi_k}^1 M_{\frac{1}{2}}(\xi) & = & M_{\frac{1}{2}}(\xi+v_k)-M_{\frac{1}{2}}(\xi)
 =   (1+a(x_j)(\xi+v_k)^2)^{-1/2} -  (1+a(x_j)\xi^2)^{-1/2},
\end{eqnarray*}
 gives
\[
 |\Delta_{\xi_k}^1 M_{\frac{1}{2}}(\xi) |\le C(1+|\xi|)^{-2},
\]
and similarly for higher order differences,
and thus $M_{\frac{1}{2}}(\xi) \in S_{\rho, \delta}^{-1}(\mathbb{T}^n \times \Z^n)$
\cite{M.Ruzhansky.V.Turunen2006, V.TurunenG.Vainikko1998}.
That is to say that $S_j$
is a periodic elliptic PsDO of order $m=-1$ with period $\frac{4\pi}{K}$ on each of the coordinate directions (using the binomial theorem and the Leibnitz formula for differences~\cite{M.Ruzhansky.V.Turunen2006} one can deduce the similar relation for any $\alpha$).
The previous statement implies that the distribution kernel of $S_j$, let us
denote it here by $K(x, y)$, has singularities at $x = y+\frac{2\pi}{K}k$, where $k\in\Z^n$. Therefore, when it is restricted to $x \in \supp(g_j )$,
$y \in \supp(g_j)$, then the singularities are contained in the set $x = y$.

Then we show that
$
 P_j = g_j(x) S_j(x, \xi) g_j(x)
$
is a PsDO for all $j\in J$.
Here we consider $g_j(x)\in C^{\infty}_0(D_j)$, $j \in J$, and thus $ g_j$ is a PsDO of order $0$.
Then it follows from \cite{LarsHormenderIII, MWWong1991} that the composition
$ S_j g_j $ is a periodic elliptic PsDO of order $-1$, and so is the operator (since $S_j$ is a PsDO of order $-1$)
$P_j$ with principal symbol
$\sigma_j(x, \xi) = g_j(x) a_j(\xi) g_j(x).$
Since $g_j(x)\in C_0^{\infty}(D_j)$, it can be viewed as $g_j(x)\in C_0^{\infty}(\Omega)$. Thus  for all $j \in J$ the symbols $\sigma_j(x, \xi)$ can be extended to $\sigma_j \in S^{-1}_{\rho, \delta}(\Omega \times \Z^n)$.

We have
$
P = \sum_{j\in J} g_j S_j g_j.
$
Here  $g_jS_j g_j$, $j\in J$ are periodic elliptic PsDOs of order $-1$.
Let us denote the principal symbol of $S_j$ by $a_j$.
So it follows that $P$ is a periodic elliptic PsDO of order $-1$ since $P$ is a sum of PsDOs of order $-1$.
\end{proof}

Since $P$ is an elliptic PsDO, it follows from \cite{LarsHormenderIII,
  MWWong1991} and references there in that there is a parametrix $Q$
such that $QP = I + R$ with $R$ a smoothing operator. For our purposes
it is sufficient that $R$ is a PsDO of order $-1$ (see PsDO
literature, parametrix, for example \cite[page 18]{SAPG}). Let us now
discuss the kernel of the operator $P$. The kernel of $P$ is contained
in the kernel of $QP$. The operator $R$ is compact. We work with
operators on $L_2$, so it is an operator from $L_2$ to $H^1$, and
therefore compact as an operator on $L_2$.  But since this is PsDO
theory, we can also work on operators on $H^s$, and then $R$ is
continuous $H^s \rightarrow H^{s+1}$ and compact as an operator on
$H^s$.
From the theory of compact operators, it follows that, for any $C >
0$, there is at most a finite dimensional set of vectors in $L_2$,
such that $\|Ru\|_{L_2} \ge C \|u\|$ (e.g. set $C = 1/2$.) Then it
shows also that the kernel of $P Q$ is finite dimensional and hence
that the kernel of $P$ is finite dimensional.
Similarly the cokernel of $P$, which is defined as the kernel of $P^∗$
is finite dimensional.
We claim that the elements of $\ker(QP )$ are in $C^\infty$. The elements of
$\ker(QP )$ satisfy  $u = -Ru.$
Let us say $u \in L_2$. Then $Ru \in H^1$ because $R$ is a PsDO of order $-1$. But
$u = -Ru$, so in fact $u \in H^1$. But then $Ru \in H^2$ , because $R$ is continuous
$H^1\rightarrow  H^2$. And so forth. It follows that $u \in C^{\infty}.$
Since we already proved in Proposition~\ref{prop:P_Hhalf_inv}
that $P$ is invertible $H^{-1/2} \rightarrow H^{1/2}$,
it follows that the kernel and cokernel of $P$ are the zero sets and that
$P$ is invertible $H^{s} \rightarrow H^{s+1}$, $s \in \R$.

Next we study the inverse $P^{-1}$ and show that it is a PsDO.
As $P$ is an elliptic PsDO, there exists $Q\in
\operatorname{Op}(S^{1})$ \cite{LarsHormenderIII, MWWong1991} such that
$  QP = I+R$
and
$
  PQ= I + \tilde R,
$
where $R \in \operatorname{Op}(S^{-\infty})$ and $\tilde R \in
\operatorname{Op}(S^{-\infty})$.
Let us consider
$
 V:H^{1/2} \rightarrow H^{-1/2}
$
be the inverse of $P: H^{-1/2}\rightarrow H^{1/2}$.
Then one gets
\begin{equation}\label{f:psinv01}
 QPV = (I+R)V
\end{equation}
and
\begin{equation}\label{f:psinv01a}
 QPV = Q.
\end{equation}
Combining (\ref{f:psinv01}) and (\ref{f:psinv01a})
we have
$
 (I+R)V=Q
$
and so
\begin{equation}\label{f:qsr001}
 Q-V=RV.
\end{equation}
Also one can write
\begin{equation}\label{f:psinv02}
 VPQ = Q
\end{equation}
and
\begin{equation}\label{f:psinv02a}
 VPQ = V(I+\tilde R).
\end{equation}
Combining (\ref{f:psinv02}), (\ref{f:psinv02a}) and (\ref{f:qsr001}) one gets
\[
Q-V=V\tilde R = (Q+(V-Q))\tilde R= Q\tilde R - RV \tilde R\in OP(S^{-\infty}),
\]
and so $V \in OP(S^{1})$.

From here we may conclude that Theorem~\ref{f:theorem001} is proved.
However, there is no estimate of $c_1$, so we add a slightly more direct
estimate for the constant $c_1$.
\begin{proposition}\label{f:theorem009a}
There are constants $0 < A \le B$ such that
\begin{equation} \label{eq:prop_final_estimate}
A \|Pu\|_{H^1} \le \|u\|_{L_2} \le B \|Pu\|_{H^1}, \quad \text{for any}\quad u\in L_2(\Omega).
\end{equation}
\end{proposition}
\begin{proof}
The second inequality follows from the fact that $P$ is a pseudodifferential
operator of order 1. For the first inequality consider the family of
pseudodifferential operators $E_{\lambda,s}$, with symbol
$
 E_{\lambda, s}(x,\xi) = \left(1+\frac{\xi^2}{\lambda^2}  \right)^{s/2}.
$
Clearly $Pu$ satisfies
\begin{equation} \label{eq:final_estimate_1}
\begin{split}
  \| P u \|_{H^1}
  \ge {}& \| E_{\lambda,1/2} P u \|_{H^{1/2}}
  = \left\| P E_{\lambda,1/2} u + [ E_{\lambda,1/2}, P ] u \right\|_{H^{1/2}}
\\
  \ge {}& \left\| P E_{\lambda,1/2} u \right\|_{H^{1/2}}
    - \left\| [ E_{\lambda,1/2}, P ] u \right\|_{H^{1/2}} .
\end{split}
\end{equation}
The first term can be estimated by
\[
  \left\| P E_{\lambda,1/2} u \right\|_{H^{1/2}}
  \ge \sqrt{\lambda} \| P \|_{H^{-1/2} \rightarrow H^{1/2}} \| u \|_{L^2} .
\]
The second term can be written as
\begin{equation} \label{eq:comm_Elambda_P}
  \sum_j [ E_{\lambda,1/2}, g_j S_j g_j ]
  = \sum_j [ E_{\lambda,1/2}, g_j ] S_j g_j + g_j S_j [ E_{\lambda,1/2}, g_j ]
\end{equation}
The commutator $[ E_{\lambda,1/2}, g_j ]$ is a pseudodifferential operator
whose symbol can estimated in terms of $\lambda$, it is given by an
asymptotic sum of the form
\[
  \sum_{j=0}^\infty \frac{1}{\lambda^{j+1}} B_j(x,\frac{\xi}{\lambda})
\]
where the $B_j$ are symbols of order $-1/2 - j$. It follows that we have
the estimate
\begin{equation} \label{eq:comm_Elambda_g}
  \left\| [ E_{\lambda,1/2}, g_j ] \right\|_{H^s \rightarrow H^s}
  \le C \lambda^{-1} .
\end{equation}
Equations (\ref{eq:comm_Elambda_P}) and (\ref{eq:comm_Elambda_g}) show
that the second term on the r.h.s.\ of
(\ref{eq:final_estimate_1}) can be estimated by $C \lambda^{-1} \| u \|_{L^2}$.
By choosing $\lambda$ large enough, it follows that the first inequality
of (\ref{eq:prop_final_estimate}) holds.
\end{proof}

\section{Numerical experiments}\label{Section4:fff}
Now the question arises how efficient it is to solve constant and variable coefficient PDEs using the proposed preconditioner? Does the preconditioned system converge faster than the unpreconditioned system while using some iterative solvers for linear system of equations, for example, the conjugate gradient method?
To investigate it, we experiment our technique focusing elliptic BVPs and compare it with the conjugate gradient method within the framework of \mat.

\subsection{One dimensional examples}
Let us start with a one dimensional case ($d=1$). Then the BVP (\ref{f:ellipticbvp01_a})
takes the form
\begin{equation}\label{f:ellipticbvp01_aa}
-\frac{d}{dx}\left(a(x)\frac{du(x)}{dx}\right)+ b(x) u(x)= f(x), \quad \forall\quad x\in\Omega\subset \R
\end{equation}
with
\[
 u=0\quad \text{on}\quad \partial\Omega
\]
where $a(x),$ and $b(x)$ are known functions of $x$ with $a(x) \ne 0.$ We approximate the problem $\mathcal{L}u=f$ using a finite difference scheme.  Let us define the grids as $\hat{x}_j=jh$, $j=0, 1, 2, \cdots,N$ with grid spacing $h$,  and
$N\ge 2$.
We define the approximations to the  solution $u$  by  $u_j$ defined on the grid points $\hat{x}_j$ for all $j=0, 1, 2, \cdots,N$.
The BVP (\ref{f:ellipticbvp01_aa}) can  be approximated by
\begin{equation}\label{f:discrete_CDS001}
-\frac{1}{h^2}\left(a_{i+\frac{1}{2}} u_{i+1}- u_{i}(a_{i+\frac{1}{2}}+a_{i-\frac{1}{2}}) + a_{i-\frac{1}{2}}u_{i-1}\right) + b_{i} u_i = f_i,
\end{equation}
for all $i=1, 2, \cdots, N$ where $a_i=a(\hat{x}_i)$, $u_i=u(\hat{x}_i)$,
$b_i=b(\hat{x}_i)$, $f_i=f(\hat{x}_i)$,
$a_{i \pm \frac{1}{2}}=\frac{a_i + a_{i \pm 1}}{2}$ and using boundary condition
$u_0=u_N=0$.
We write (\ref{f:discrete_CDS001}) in the matrix form as
$
 \mathcal{A}u=f.
$
In Figure~\ref{fig:cndntnnumber}, we  plot spectral radius and condition numbers of $\mathcal{A}$ for several choices of system size to demonstrate the
polynomial growth of condition numbers.
\begin{figure}[ht]
\begin{center}
  \includegraphics[width=0.85\textwidth,height=9.5cm]{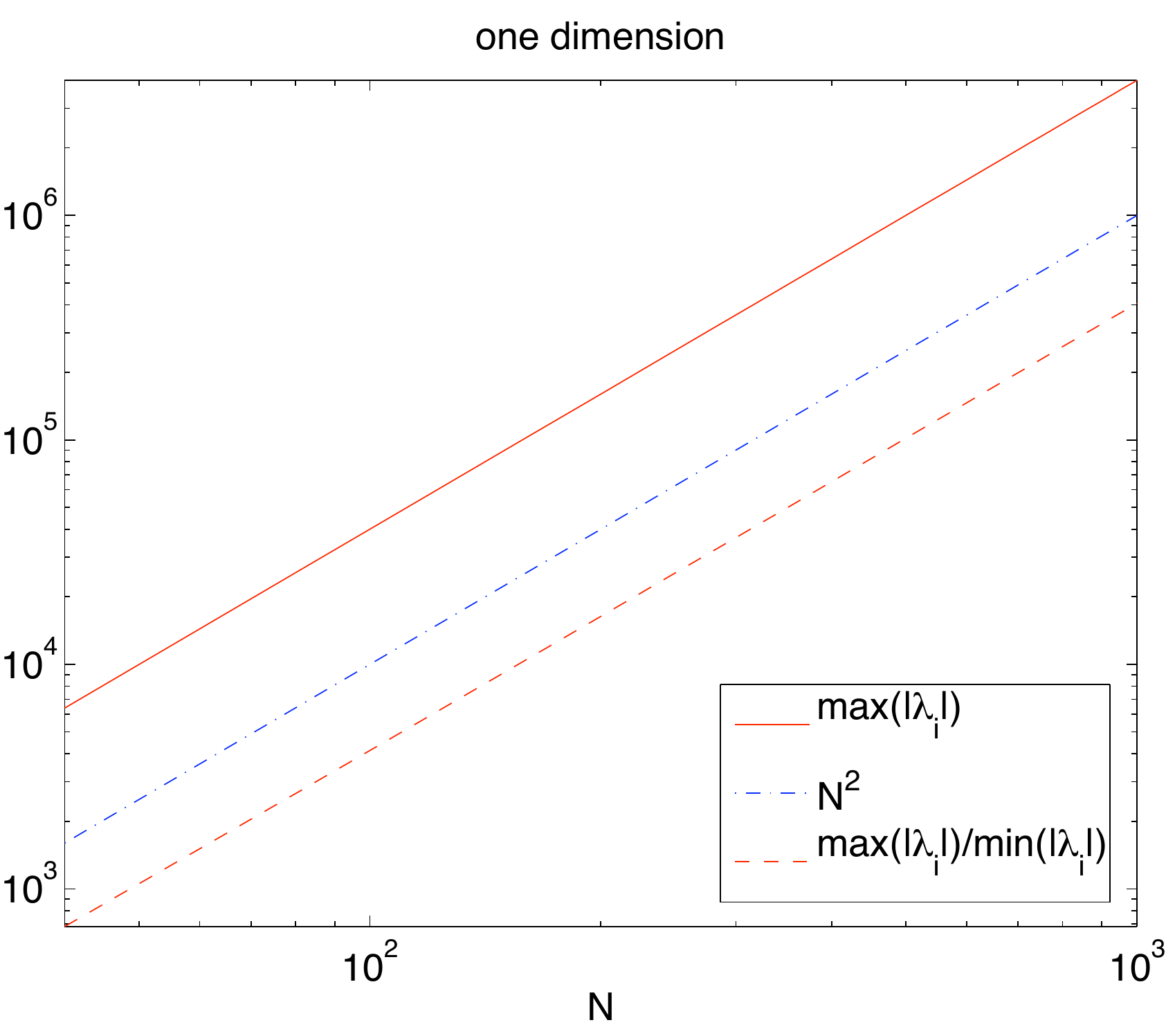}
\end{center}
\caption{We show spectral radius, and condition number of the matrix operator acting on (\ref{f:discrete_CDS001}) for various choices of system sizes with $a=1$ and $b=0$.}
    \label{fig:cndntnnumber}
    \end{figure}
It is noticed from Figure~\ref{fig:cndntnnumber} (also several articles and books on finite difference schemes have a clear discussion about the eigenvalues of $\mathcal{A}$, e.g., \cite{NHA001, CDM001, AQRSFS01} and references therein) that  the condition number of a matrix $\mathcal{A}$  grows like $\mathcal{O}(1/h^2)$ (when $a(x)=1$),  and as a result the convergence of any iterative methods become slower.

We consider three examples to address the questions arise in our
previous discussion. In the first example, we consider $a(x)=1$,
$b(x)=0$ and observe singular values and condition numbers of the
symmetric and left preconditioned matrices. We also compare the number
of iterations taken by the preconditioned solvers to converge.  Then
we repeat our experiments when $a(x)$ varies with $x$ in the second
example. Here we compare condition numbers by varying number of
windows (to cover the domain) to show the efficiency of windowing. We
also show the advantage of using the preconditioned CG solver by
comparing the number of iterations with the CG solver. In the third
example the coefficient varies discontinuously.

\begin{example}\label{f:example_constant_coeff01}
Consider the BVP
\[
  -\frac{d^2}{dx^2}u(x)=f(x), \quad \forall \quad 0< x< 1,
\]
with boundary conditions
\begin{enumerate}

\item $u(0)=0$ and $u(1)=0$

\item periodic boundary condition $u(0)=u(1)$.
\end{enumerate}
The discrete operator $\mathcal{A}$ can be found in (\ref{f:discrete_CDS001}),
to enforce the periodicity, when needed, we define $\mathcal{A}(N,1)=\mathcal{A}(1,N)=-\frac{1}{h^2}$.
In Figure~\ref{fig:singlarvalues_of_PAP_PAa} 
we compare singular values of $P\mathcal{A}P$ and $P\mathcal{A}$ with the singular values of $\mathcal{A}$ (with periodic and non-periodic settings).
\begin{figure}[ht]
\begin{center}
\includegraphics[width=0.49\textwidth,height=7.5cm]{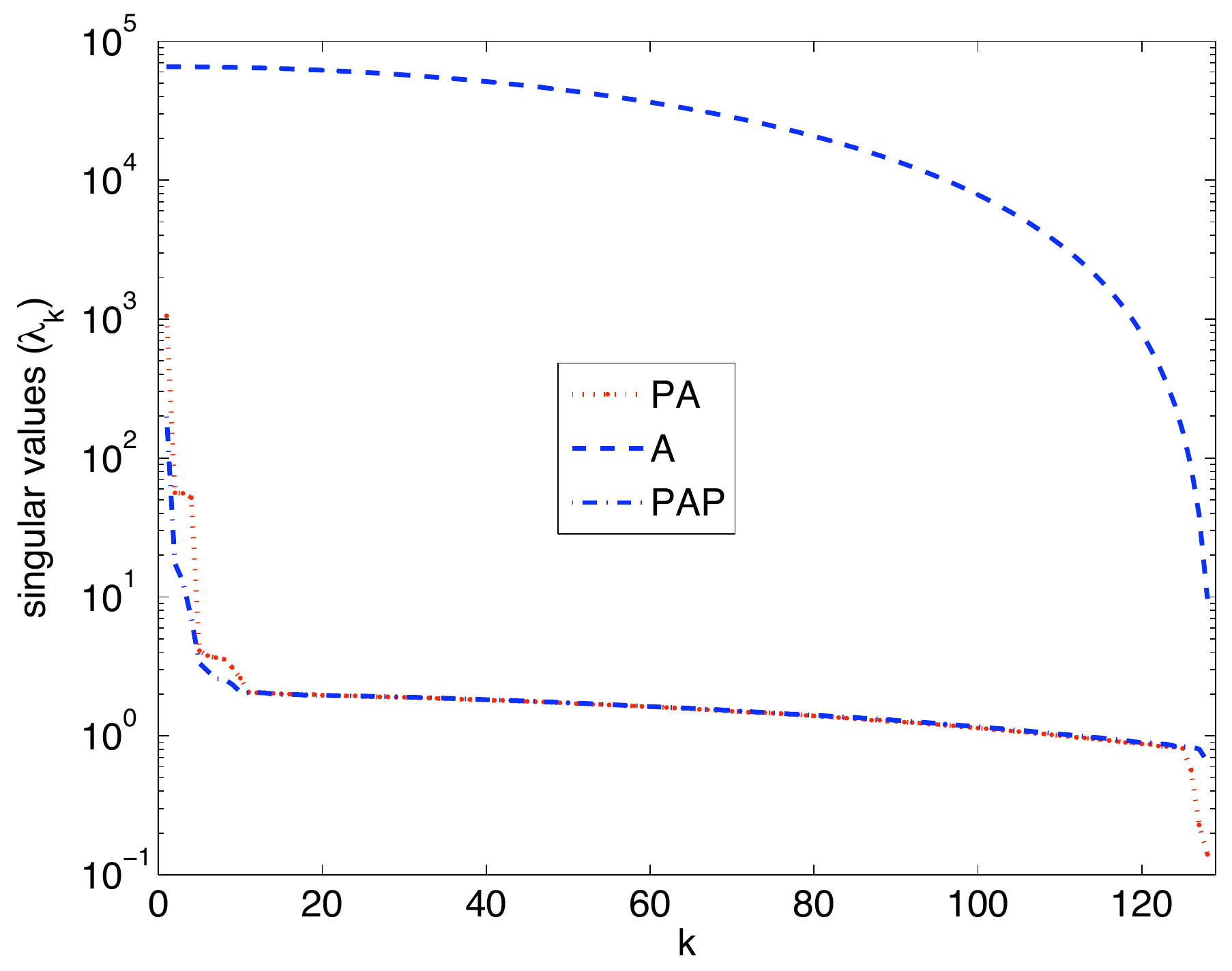}
\includegraphics[width=0.49\textwidth,height=7.5cm]{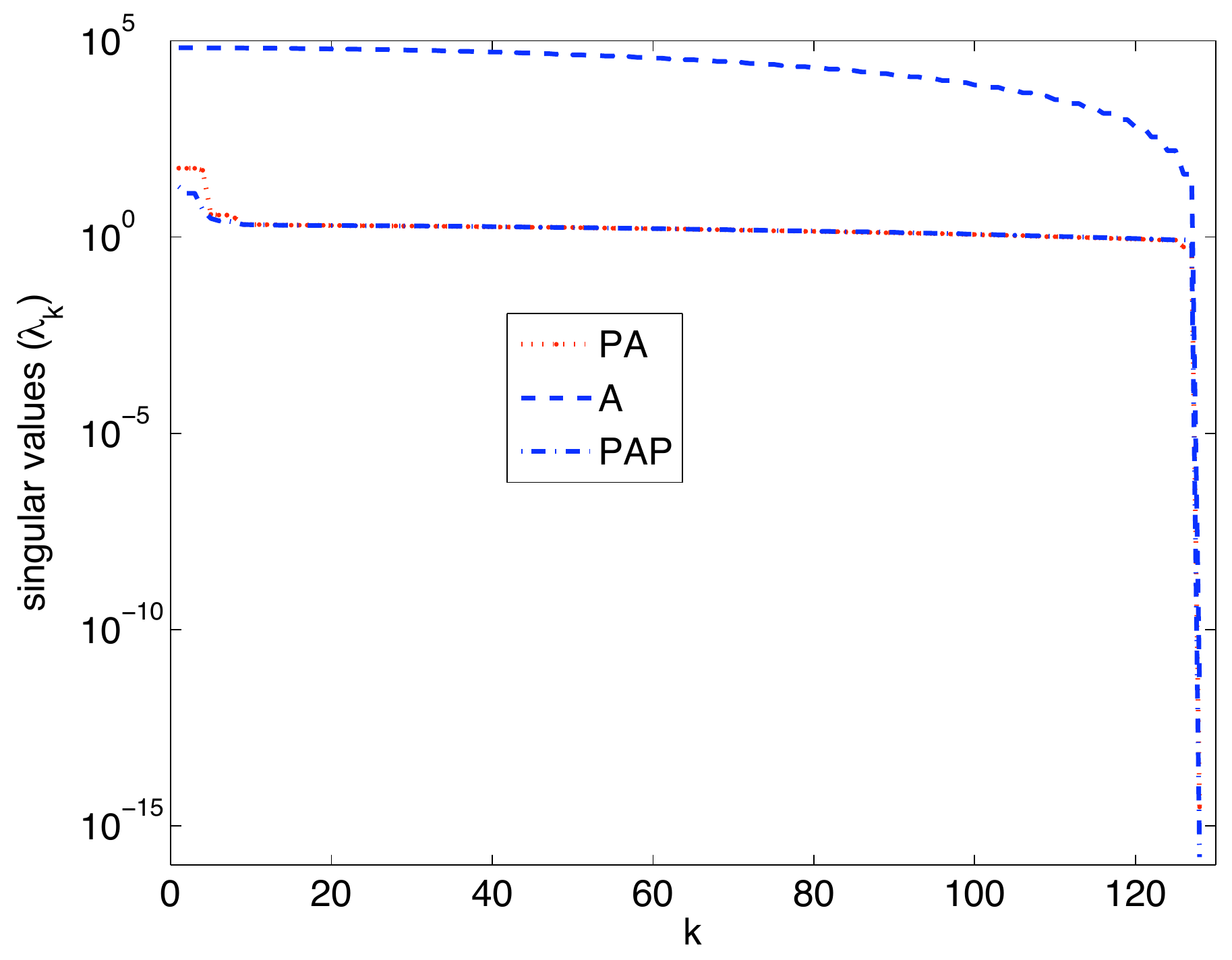}
\end{center}
\caption{Here we plot singular values of the preconditioned operators (non-periodic operator (left figure)
and periodic operator (right figure))  where $K=4$ and $N_l=2^6$. From here  $A$ appearing on the figures represents the discrete operator $\mathcal{A}$.}
\label{fig:singlarvalues_of_PAP_PAa}
\end{figure}
From these computations,  we notice that
most of the singular values are close to one, whereas the singular values / eigenvalues of $\mathcal{A}$ are  \cite[page 561]{CDM001}
\[
\lambda_j=\frac{4}{h^2} \sin^2\left(\frac{j\pi}{2(N+1)}\right), \quad \forall \quad j=1, 2, \cdots, N.
\]
In the periodic case there is a singular value zero due to the constant
solutions.

In Figure~\ref{fig:cndntnnumberall22} and Figure~\ref{fig:cndntnnumberall_1a}, we show  the condition numbers of the symmetric and the left preconditioned operators considering $\mathcal{A}$ as a periodic and a non-periodic operator respectively.  We notice that the condition numbers of $P\mathcal{A}P$ and $P\mathcal{A}$ appear to be bounded when $\mathcal{A}$ is periodic. The condition number of $P\mathcal{A}P$ grows as $\mathcal{O}(N)$ whereas condition number of $P\mathcal{A}$ grows similar to the unpreconditioned operator $\mathcal{A}$ when $\mathcal{A}$ is non-periodic.
We also notice from Figure~\ref{fig:cndntnnumberall_1a} that
$\frac{\lambda_3}{\lambda_{N-2}}$ grows much slower or hardly at all for both
the Fourier frame preconditioned operators $P\mathcal{A}P$ and
$P\mathcal{A}$ (non-periodic operators), whereas
$\frac{\lambda_3}{\lambda_{N-2}}$ grows as of
$\frac{\lambda_1}{\lambda_{N}}$ for the unpreconditioned operator
$\mathcal{A}$ ($\lambda_3$ is the third highest singular value and
$\lambda_{N-2}$ is the third lowest singular value). Thus analyzing
Figure~\ref{fig:singlarvalues_of_PAP_PAa},
Figure~\ref{fig:cndntnnumberall22}, and
Figure~\ref{fig:cndntnnumberall_1a} we expect that the left
preconditioner and the symmetric preconditioner would perform well
since most of the singular values are clustered to a constant for the
both preconditioned system (though the condition number of the left
preconditioner grows as of $\mathcal{A}$ when we consider a non
periodic boundary condition).
\begin{figure}[ht]
\begin{center}
\includegraphics[width=0.79\textwidth,height=8.5cm]{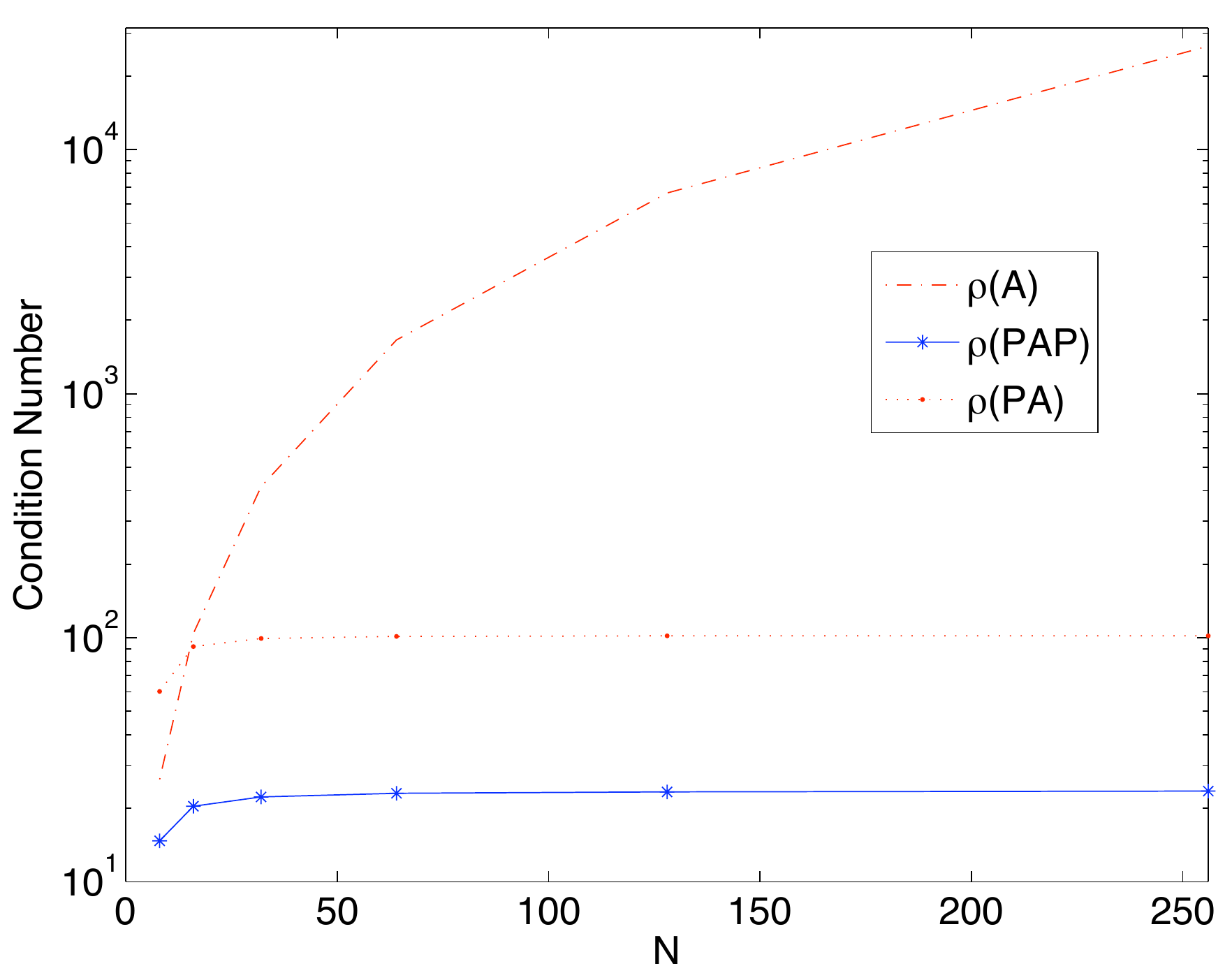}
\end{center}
\caption{Condition numbers of the operators for various choices of system sizes (periodic case).}
    \label{fig:cndntnnumberall22}
    \end{figure}
\begin{figure}[ht!]
\begin{center}
\includegraphics[width=0.49\textwidth,height=7.5cm]{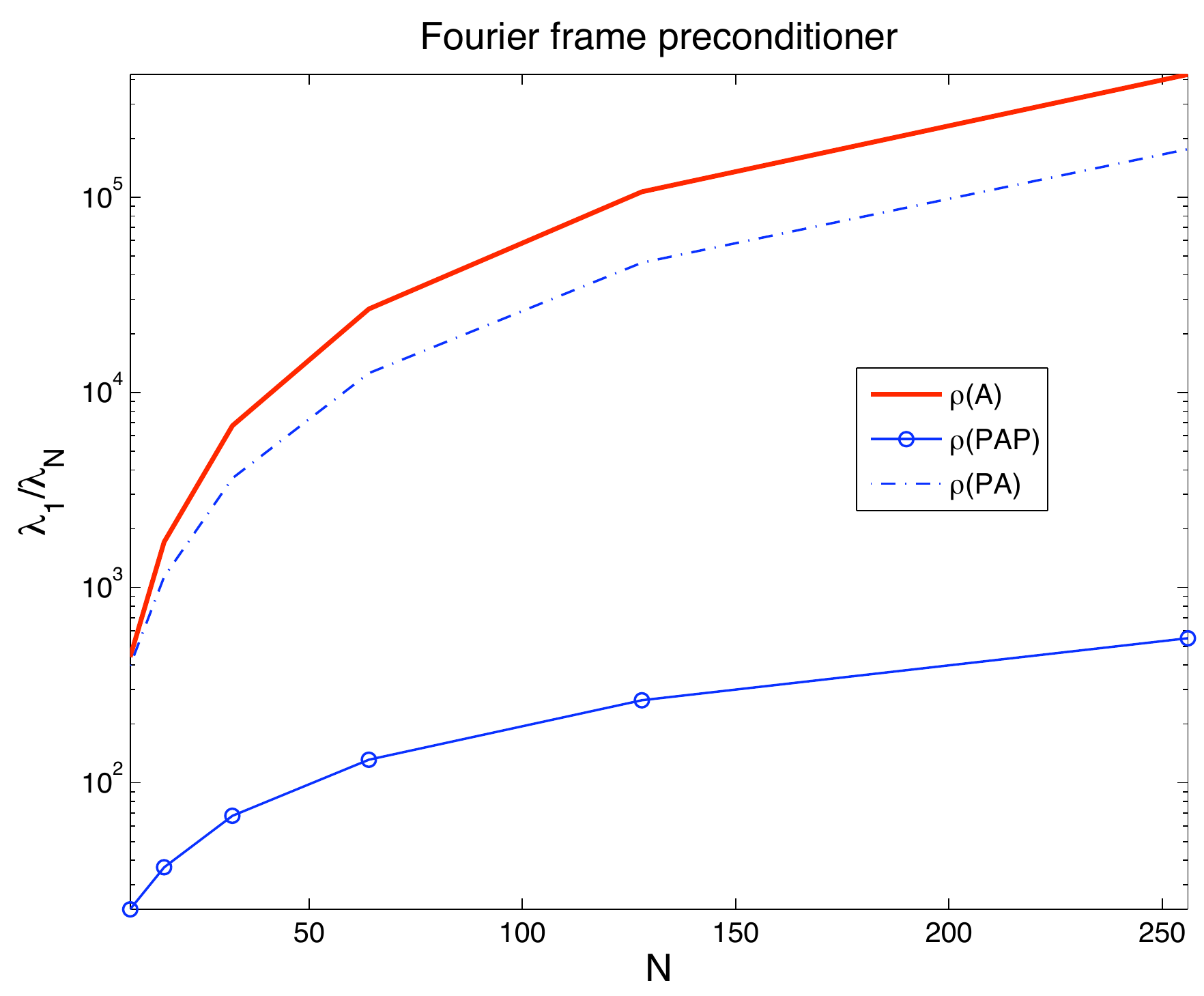}
\includegraphics[width=0.49\textwidth,height=7.5cm]{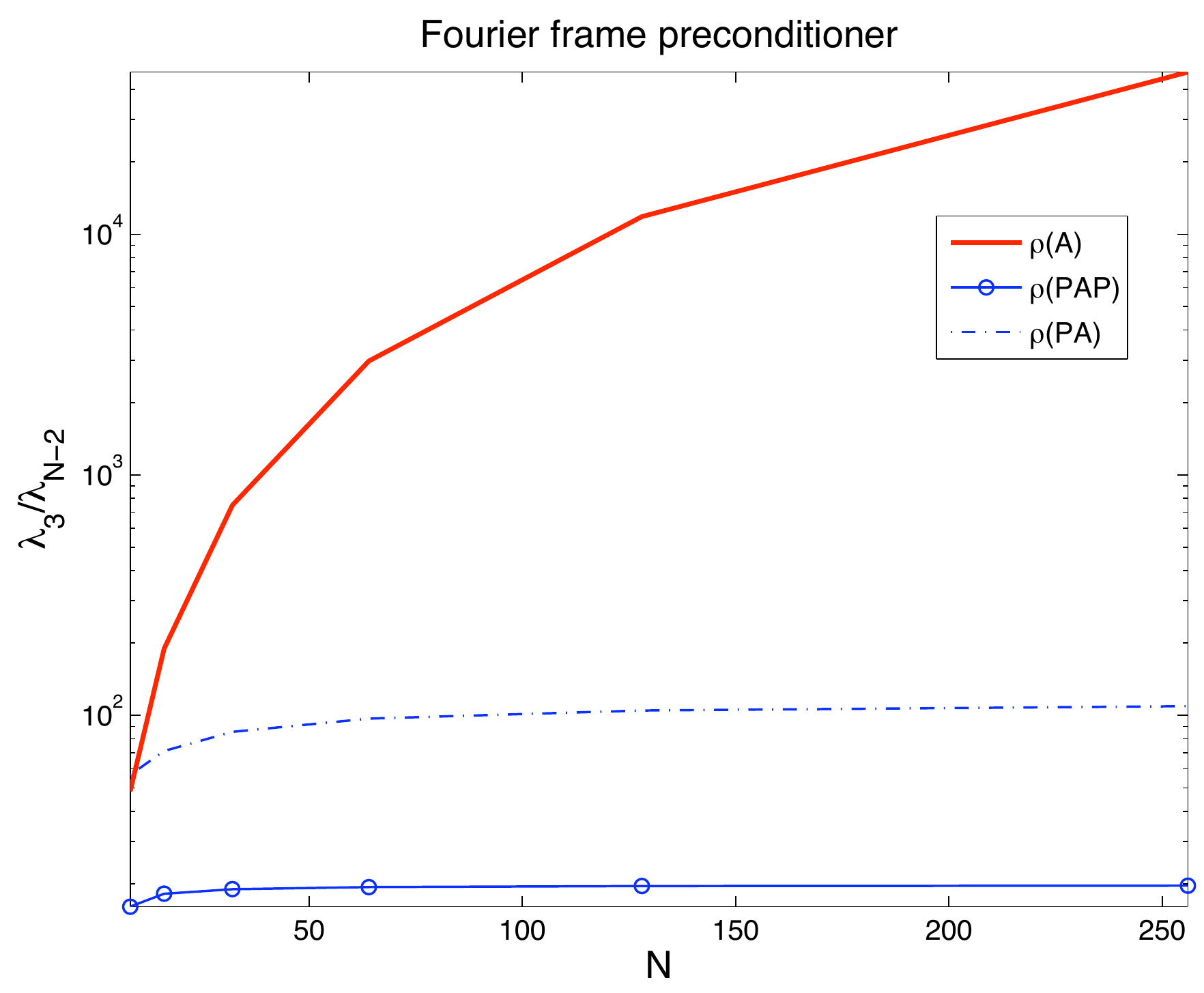}
\end{center}
\caption{Condition numbers of the  windowed Fourier frame preconditioned operators (non-periodic case)  for various choices of system size.}
\label{fig:cndntnnumberall_1a}
\end{figure}
The condition number of $P\mathcal{A}$ and $\mathcal{A}P$ are the
same, so we present experimental results of the left preconditioned
system only.

Then we focus on solving the problem using the conjugate gradient
solvers.  We present the number of iterations taken to converge for
three different choices of solvers considering
$f(x)=\exp(2\pi(x-0.5))$ in Figure~\ref{fig:1111_1a}.  From this
experiment we notice that both the SPCG and the LBICG take a very few
iterations to converge compared to the CG (for both the periodic and
the non-periodic boundary value problems).

This is to note that we have considered windowed sine frames (WSF) as
well for the same computations. It produces the similar results as of
WFF, so one can also use WSF for the computations as well.
\begin{figure}[ht!]
\begin{center}
\includegraphics[width=0.49\textwidth,height=7.5cm]{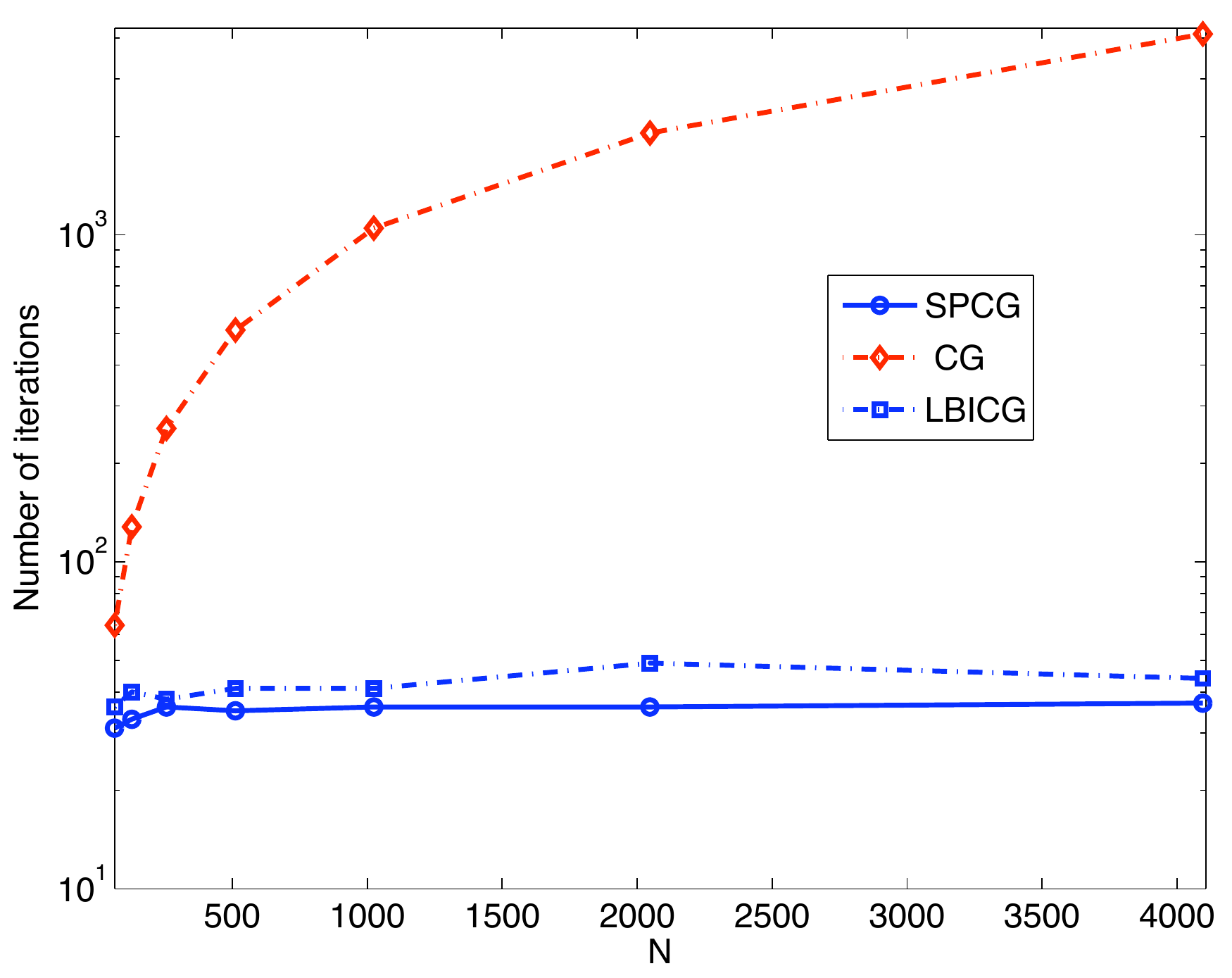}
\includegraphics[width=0.49\textwidth,height=7.5cm]{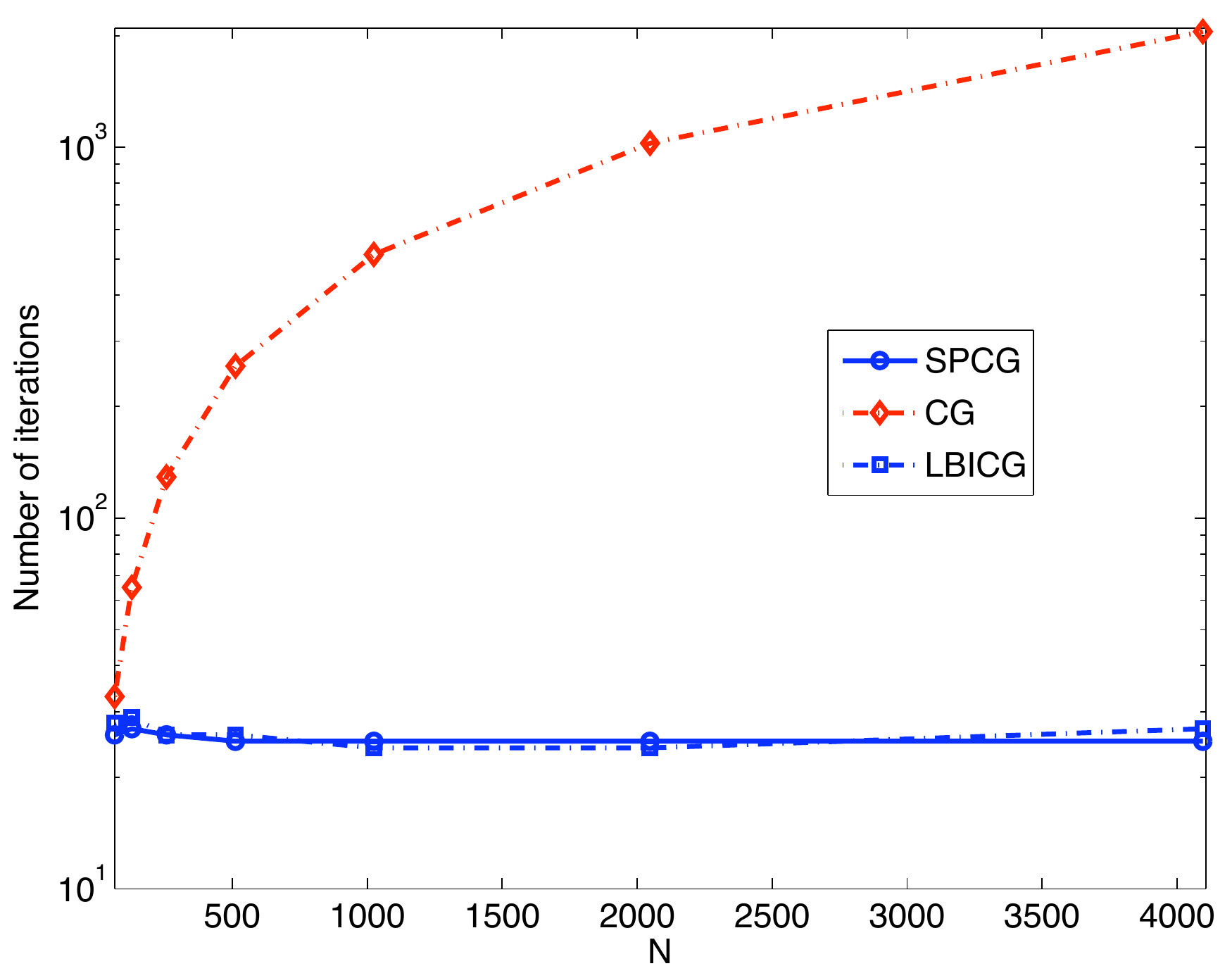}
\end{center}
\caption{Number of iterations  taken by various solvers to converge (we consider 4 windows here) considering boundary conditions  $(1)$ $u(0)=0$ $u(1)=0$ (left figure), $(2)$ $u(0)=u(1)$ (right figure).}
\label{fig:1111_1a}
\end{figure}
\end{example}

We have experimented with the preconditioner considering a simple
constant coefficient BVP in
Example~\ref{f:example_constant_coeff01}. This preconditioner is
actually designed for variable coefficient problems. In the following
example we aim to demonstrate the advantage of
windowing for such problems.

\begin{example}
Consider the periodic BVP
\[
-\frac{d}{dx}\left(a(x)\frac{du(x)}{dx}\right)+b(x) u(x)=f(x), \quad \forall \quad 0 \le x \le 1,
\]
where
$a(x)= 10 - 9.5 \cos(2 \pi x)$, $b(x)=1$  and  
\[
f(x)=
\left\{
\begin{array}{rr}
\exp(x) & \text{when} \quad 0 <x \le 0.25\\
 \exp(-x) & \text{ if} \quad 0.25<x \le 1.
\end{array}
\right.
\]

We use the scheme (\ref{f:discrete_CDS001}) and consider $tol=10^{-10}$.
In Figure~\ref{fig:cxxx_1uxxx}, we compare condition number of the discrete symmetric preconditioned variable coefficient operator in $0\le x \le 1$  by varying the number of windows to cover the domain.
From this experiment we notice that the more subdivisions of the domain (using windows) one considers, the smaller the condition number becomes.
Then we compare the number of iterations taken by the CG method and the SPCG method (considering $8$ windows to cover the periodic $[0, 1]$ domain) and
notice the good behavior of the preconditioned system.
\begin{figure}[ht]
 \begin{center}
\includegraphics[width=0.49\textwidth,height=7.5cm]{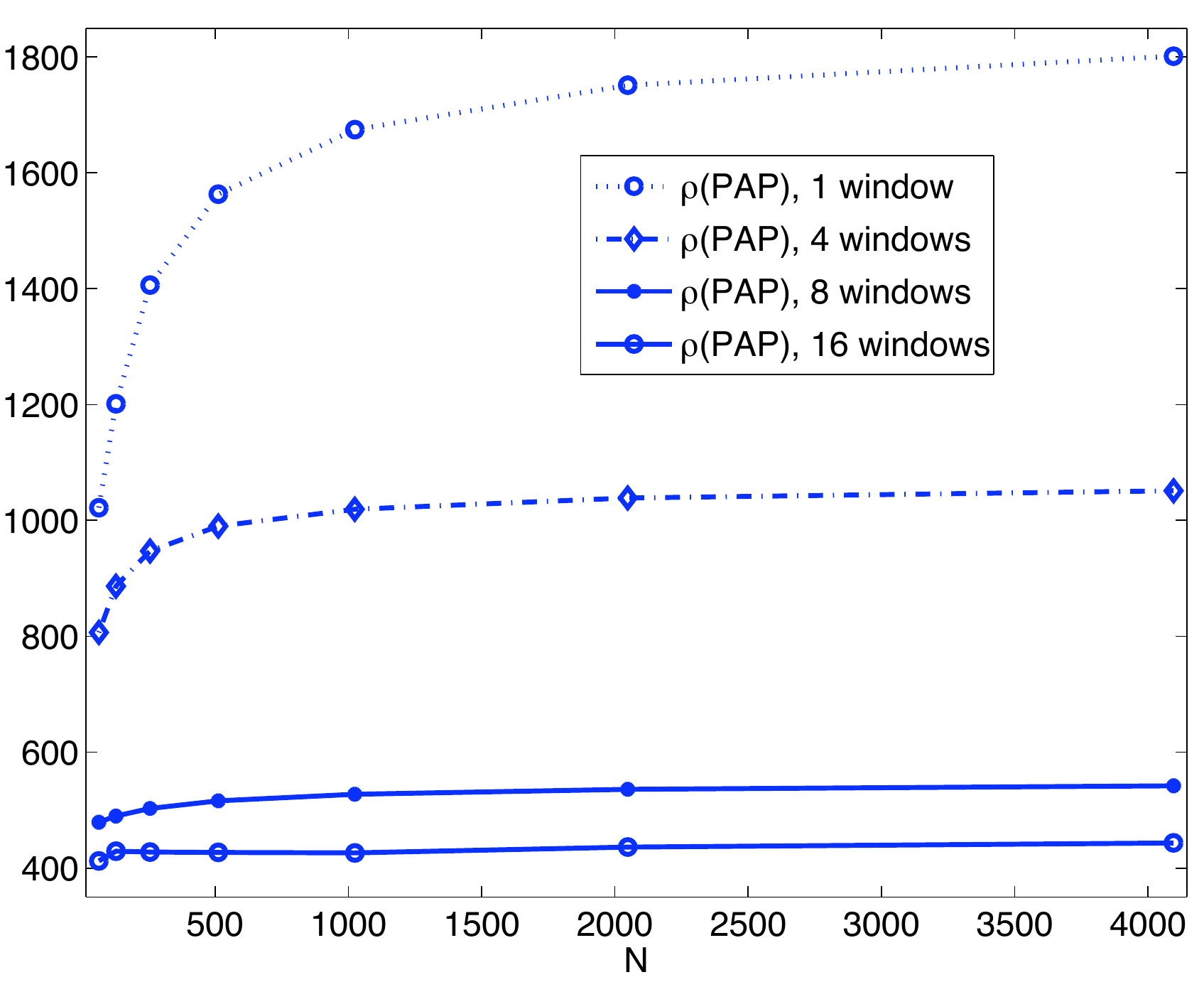}
\includegraphics[width=0.49\textwidth,height=7.5cm]{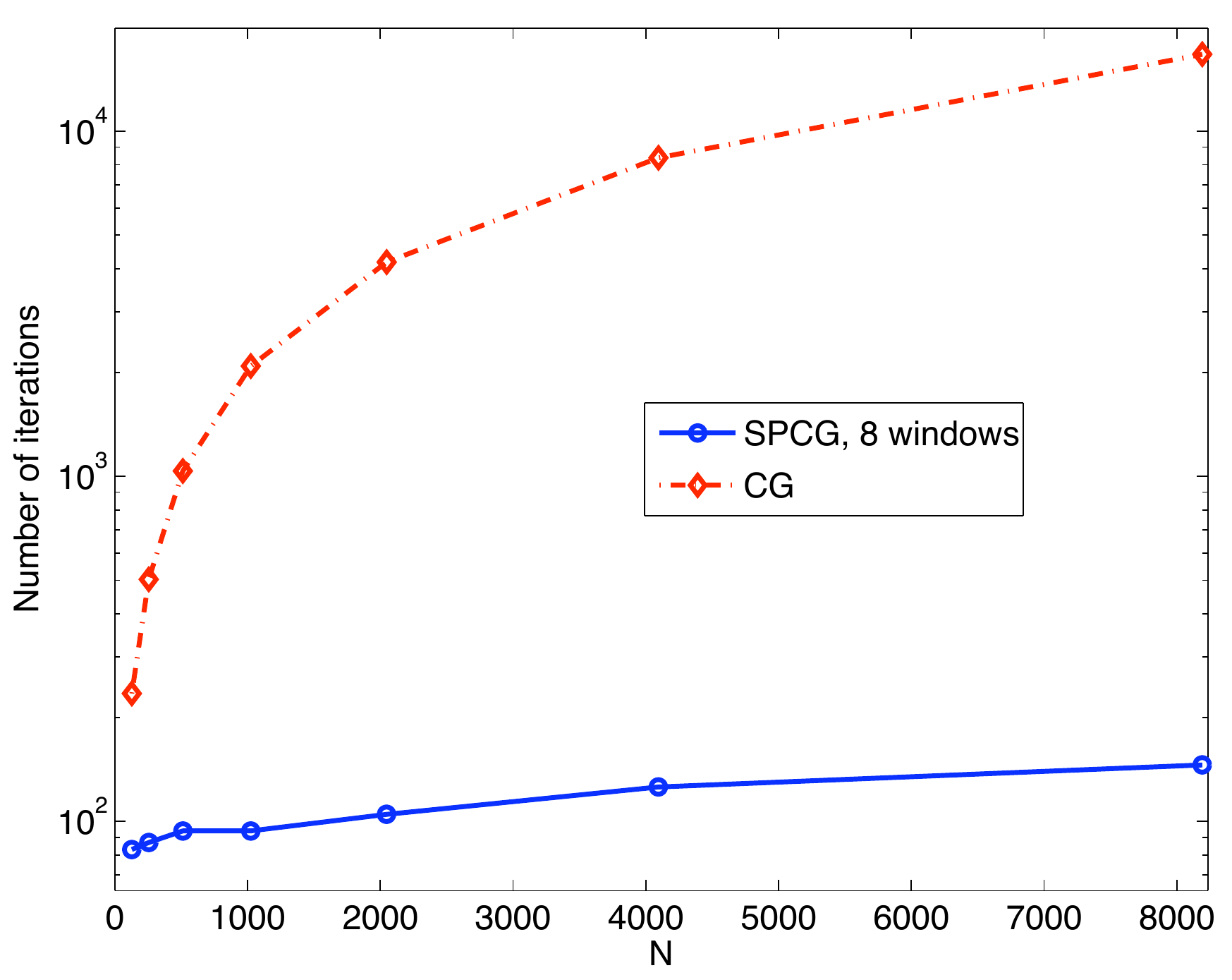}
\end{center}
\caption{We show the condition numbers for various choices of use of  number of window functions.
This left figure shows the efficiency of
windowing by reducing the condition number. The right figure shows the efficiency of windowing by reducing the number of iterations significantly.
}
\label{fig:cxxx_1uxxx}
\end{figure}
\end{example}
%
%
\begin{example}
Consider the BVP
\[
  -\frac{d}{dx}\left(a(x)\frac{du(x)}{dx}\right)+b(x) u(x)=f(x), \quad \forall \quad 0< x< 1,
\]
with boundary conditions
%
 $u(0)=0$ and $u(1)=0,$
where $b(x)=1$ for all $0<x<1$,
$
f(x)
$
as a random function and
\[
a(x)=\left\{
\begin{array}{rr}
\exp(x) & \text{when} \quad 0<x<0.5\\
\exp(-x) & \text{ if} \quad 0.5\le x<1.
\end{array}
\right.
\]
We use the scheme (\ref{f:discrete_CDS001}) and consider $tol=10^{-10}$. In Figure~\ref{fig:cxxxxxxxxxx_1ux} we notice that the condition number keeps growing
but more slowly for large $N$, while the number of iterations appears to
becomes constant or close to constant.
\begin{figure}[ht]
 \begin{center}
\includegraphics[width=0.49\textwidth,height=7.5cm]{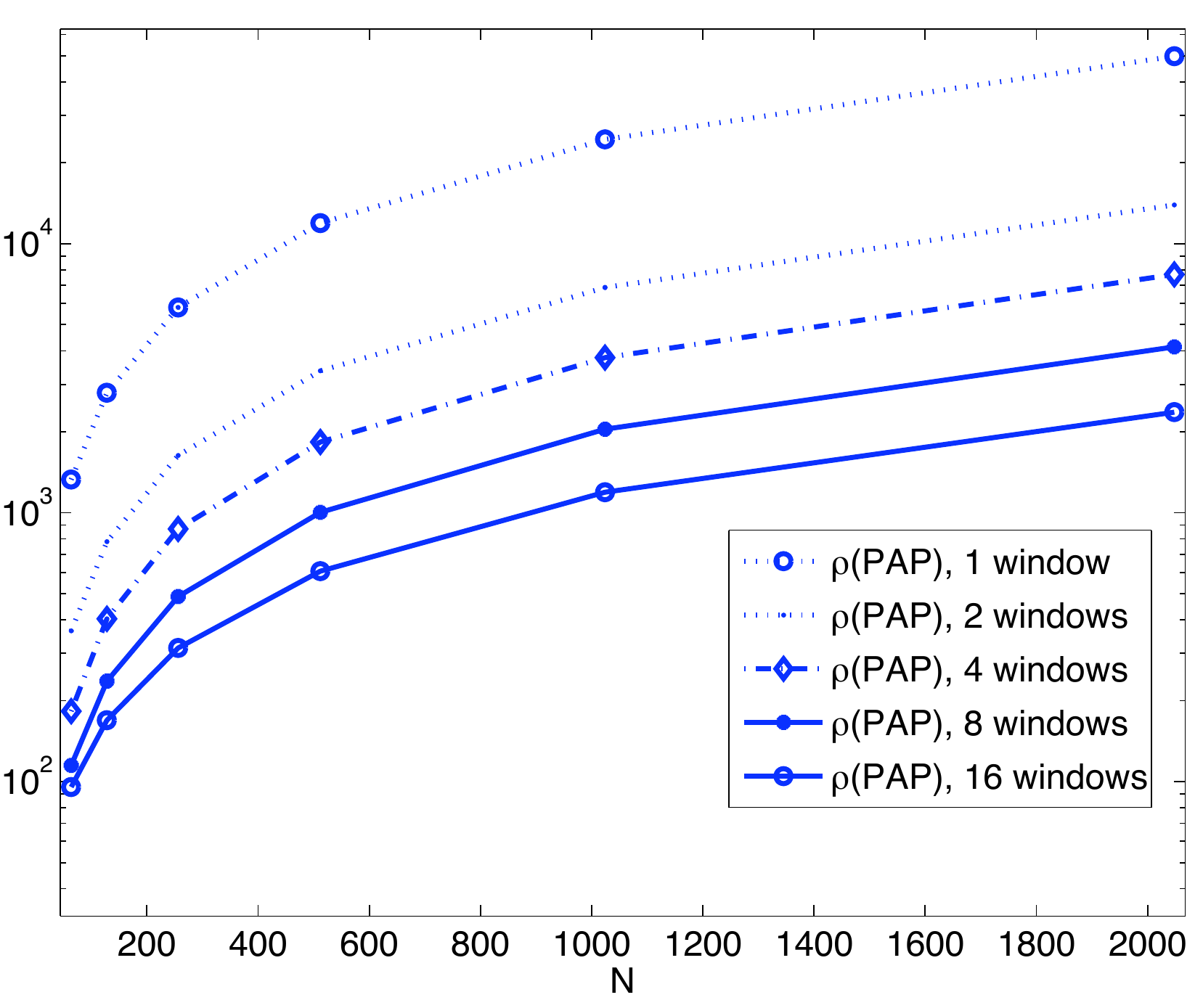}
\includegraphics[width=0.49\textwidth,height=7.5cm]{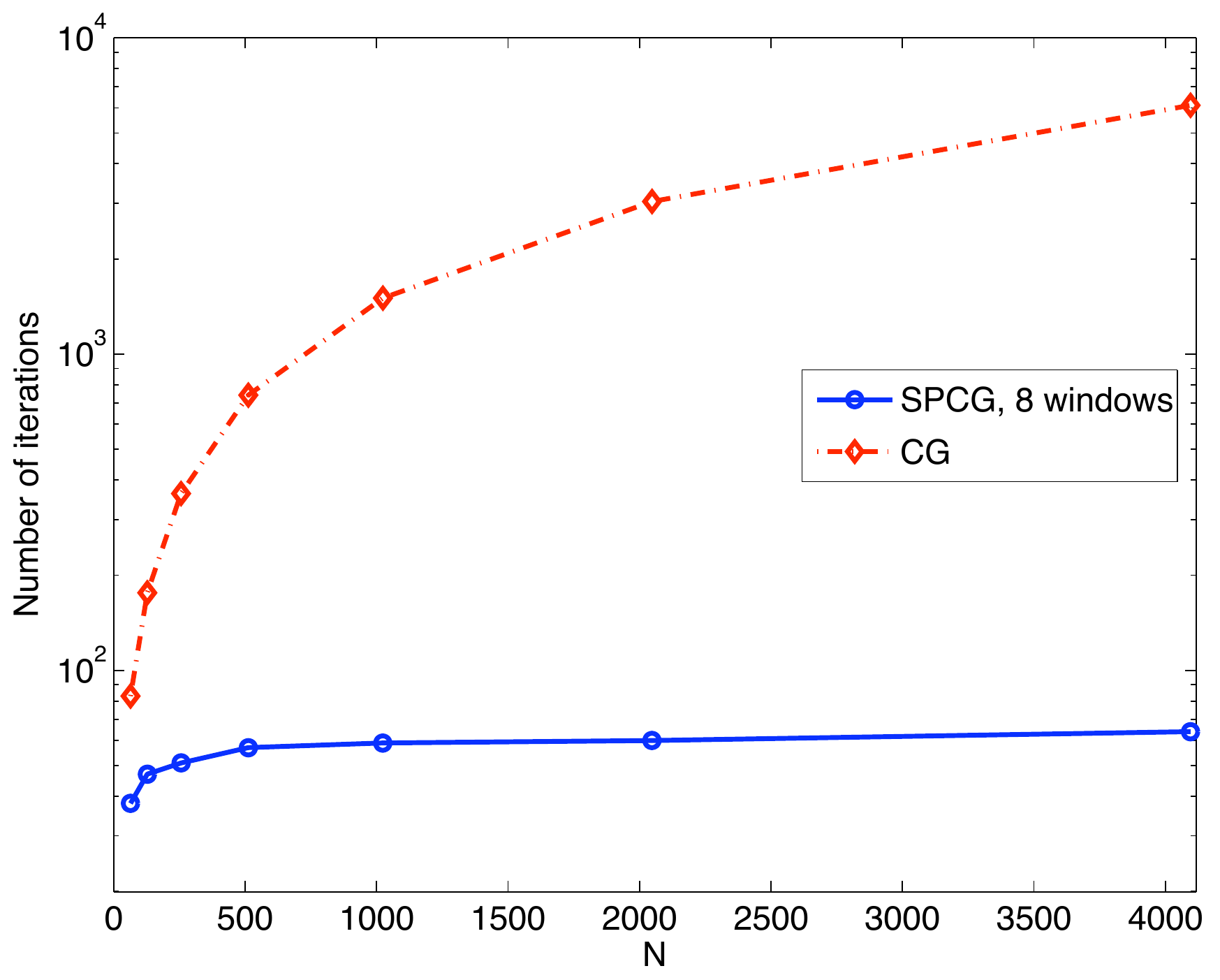} 
\end{center}
\caption{%
The left figure shows condition numbers for different choices of windows, and the right figure shows  the number of iterations taken to converge by the conjugate gradient method and the symmetric preconditioned conjugate gradient method.%
} \label{fig:cxxxxxxxxxx_1ux}
\end{figure}
\end{example}
\subsection{Two dimensional examples}\label{2dexamples}
Here we apply our technique in two dimensions.
We consider one example where the operator is very strong in one direction to demonstrate the advantage  of using the exact symbol. For the other example we consider an operator that contains strongly varying coefficient $a(x, y)$ acting on it to demonstrate further the advantage of windowing.
\begin{example}
Consider
\begin{eqnarray*}
\mathcal{L}u(x, y)&=& -10 \frac{\partial^2 u(x, y)}{\partial x^2}- \frac{1}{10}\frac{\partial^2 u(x, y)}{\partial y^2}=f(x, y),
\end{eqnarray*}
and
$\Omega \subseteq \mathbb{R}^2$ with $u|_{\partial\Omega} = 0$.
Now let $\Omega=(0, 1)^2$ be the unit square and similar to one dimensional case, we define $\nabla y=\nabla x=h=1/N>0$ and
$\hat{x}_{i,j}=(ih, jh)$, $i, j = 0, 1, 2, \cdots, N$.
The results have been compared in Figure~\ref{fig:dd_iter_cpu_FT_DFT01}. Here we notice that the windowed Fourier frame preconditioners designed with the exact symbol performs significantly better.
\begin{figure}[ht!]
\begin{center}
\includegraphics[width=0.49\textwidth,height=7.5cm]{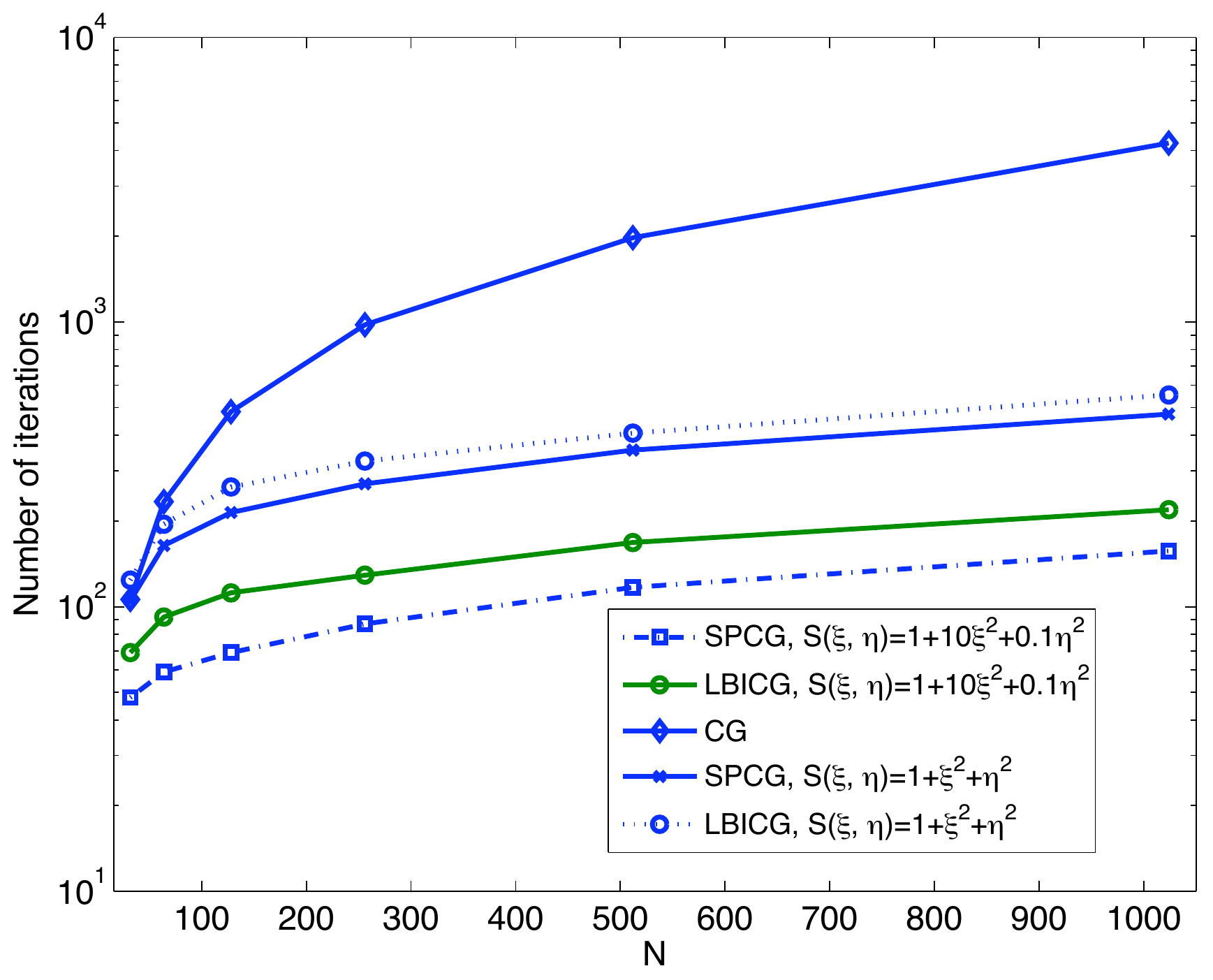}
\includegraphics[width=0.49\textwidth,height=7.5cm]{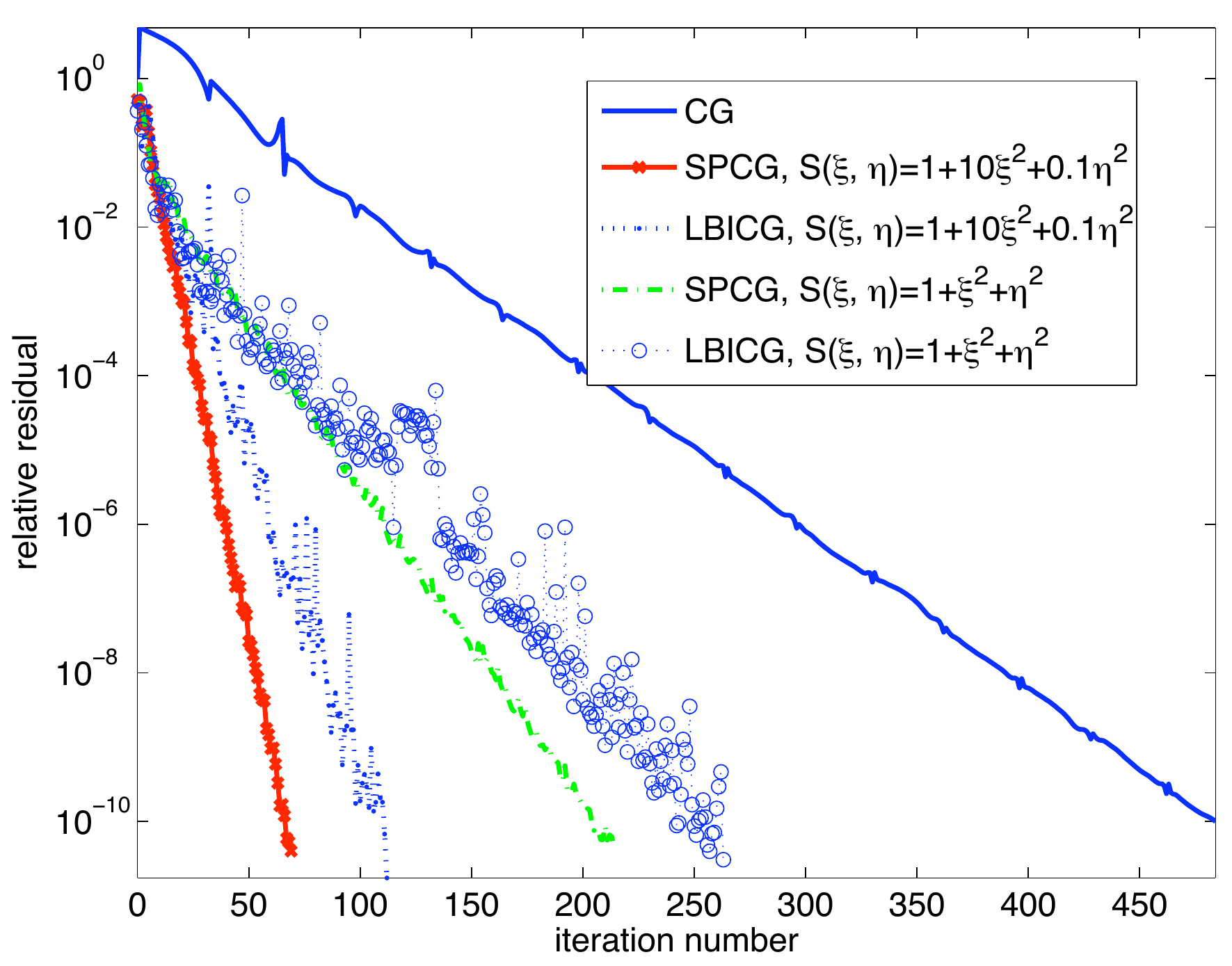}%
\end{center}
\caption{ In the left figure, we compare number of iterations considering the exact symbol $S(\xi, \eta) = 1+ 10 \xi^2 + \frac{1}{10} \eta^2$ and an isotopic symbol $S(\xi, \eta) = 1+  \xi^2 + \eta^2$. Here we use  $N_{l, x}=N_{l, y}=2^{2:7}$ (thus $N=2^{5:10}$ is the number of points on each direction).
In the right figure  we compare relative residuals for each iteration number considering $N_{l, x}=N_{l, y}=2^{4}$.
For this computation we consider $K_x=K_y=4$ windows in each directions, $f(x, y)=e^{-(x+y)}$ and  $tol=10^{-10}$.}
    \label{fig:dd_iter_cpu_FT_DFT01}
    \end{figure}
\end{example}
\begin{example}
Consider
\begin{eqnarray*}
\mathcal{L}u(x, y)&=& -\frac{\partial}{\partial x}\left(a(x)\frac{\partial u(x, y)}{\partial x}\right)- \frac{\partial}{\partial y}\left(b(y)\frac{\partial u(x, y)}{\partial y}\right)=f(x, y),
\end{eqnarray*}
in a periodic domain
$\Omega=[0, 1]^2$,  where $a(x)=10-9.5 \cos(2\pi x)$ and $b(y)=1$.
For this computations we consider symmetric preconditioned solvers only.
In Figure~\ref{fig:singlarvalues_of_iter_cpu_FT_DFT01aa}, we compare the Fourier preconditioned conjugate gradient method  (one Fourier transform only, considering no windowing) and the windowed Fourier frame preconditioned conjugate gradient method ($4$ windows in each direction). Here we observe that the preconditioned solver based on the windowed Fourier frame performs better than the Fourier preconditioned solver.
\begin{figure}[ht!]
\begin{center}
\includegraphics[width=0.8\textwidth,height=8.5cm]{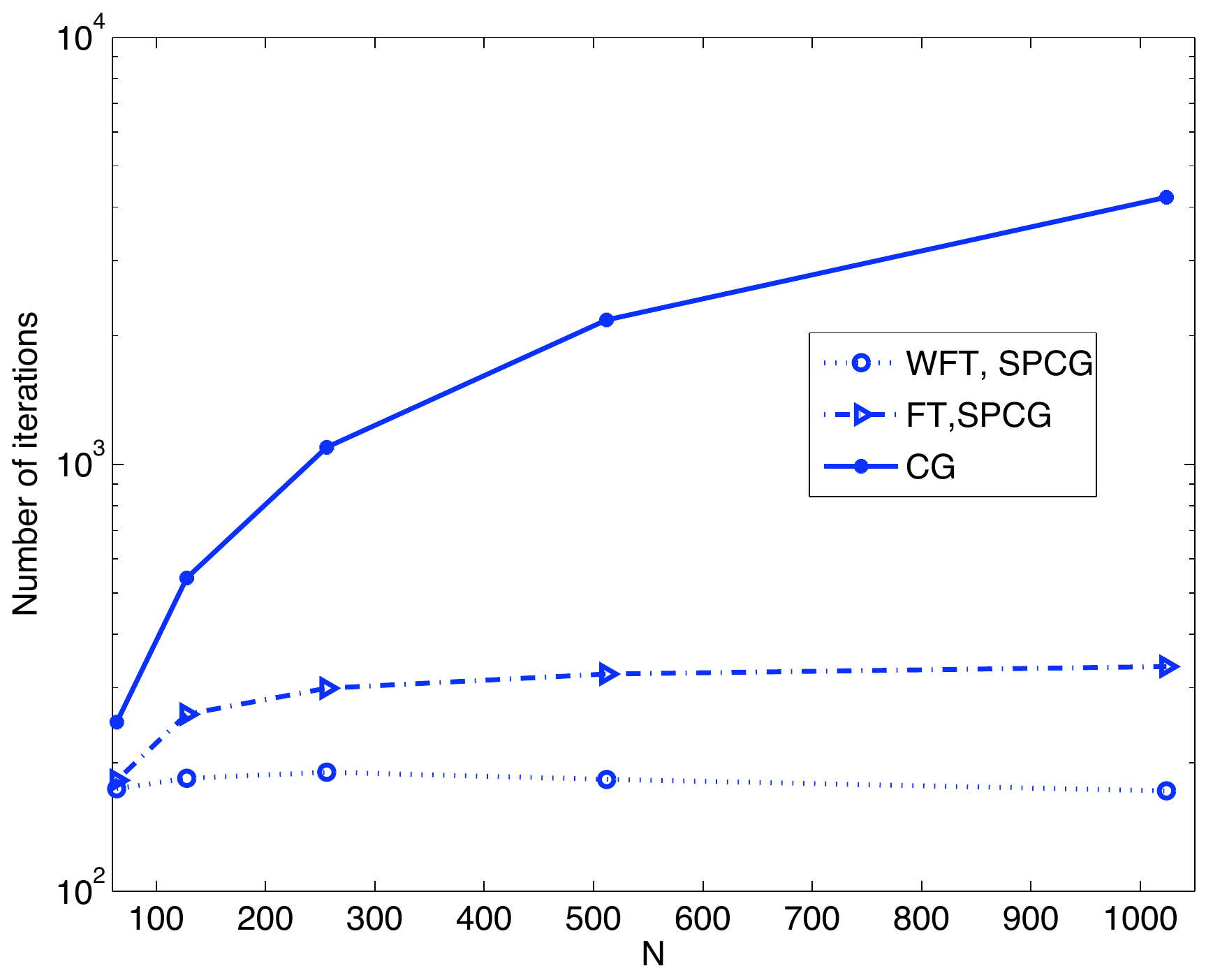}
\end{center}
\caption{ We compare number of iterations considering symmetric preconditioner with $4$ windows and without windows. We consider the symbol $S(\xi, \eta) = 1+ a(x) \xi^2 + b(y) \eta^2$. Here we use $K_x=K_y=4$, $N_{l, x}=N_{l, y}=2^{3:7}$ (thus $N=2^{6:10}$ is the number of points on each directions). For this computation we consider $f(x, y)=e^{-(x+y)}$ and  $tol=10^{-10}$.}
    \label{fig:singlarvalues_of_iter_cpu_FT_DFT01aa}
    \end{figure}
\end{example}

\section{Conclusions}\label{conclusions001}
We study windowed Fourier frames (WFF)  focusing on elliptic boundary value
problems and  present new preconditioners based on the symbol of the operator
and WFFs.
From this study we conclude that the preconditioners based on WFFs work
well for elliptic problems, most of the singular values clustered to a
constant approximately for both preconditioned operators (SPCG and
LBICG). For periodic domains, the condition number is bounded. However,
when a non-periodic domain is considered, there are a few very large
and a few very small singular values, which cause the condition number
of the symmetrically preconditioned system to grow $\mathcal{O}(N)$
whereas  the condition number of left preconditioned system grows close
to  $\mathcal{O}(N^2)$. This affects the convergence of the CG method
relatively little, as it concerns only a few singular values, but
a better definition of such a preconditioner on a bounded
domain is clearly a question for further research.

As expected both the SPCG and the LBICG take very few iterations to
converge compared to the unpreconditioned CG method. For
multidimensional PDEs, if the problem is strongly dominant in one
direction then the use preconditioners based on the exact symbols is
recommended. We notice that the use of reasonably many windows (while
defining the preconditioner) has an advantage for PDEs with strongly
varying coefficients.

\appendix

\section{Window function construction}\label{f:appendix}
The window function plays an important role in the preconditioner,
therefore we discuss its construction.
A priori, the objective is to select a window function that satisfies
the properties described in section~\ref{sec:precon_construct}
and decays rapidly in the frequency domain.
It is well known that smooth functions have fast decay for large
frequencies  \cite{Mallat2009}.
So we want to form window functions that have at least some vanishing
derivatives at both ends of the support.
In \cite{Mallat2009} there are some general rules for constructing
window functions, as well as some specific examples, we make use of this
and also compare some of the example with a construction of our own.
We present a schematic window function and two translations in
Figure~\ref{f:profilewindowfunctionsine01aaa}.
\begin{figure}[here]
     \begin{center}
     \includegraphics[width=0.85\textwidth,height=8cm]{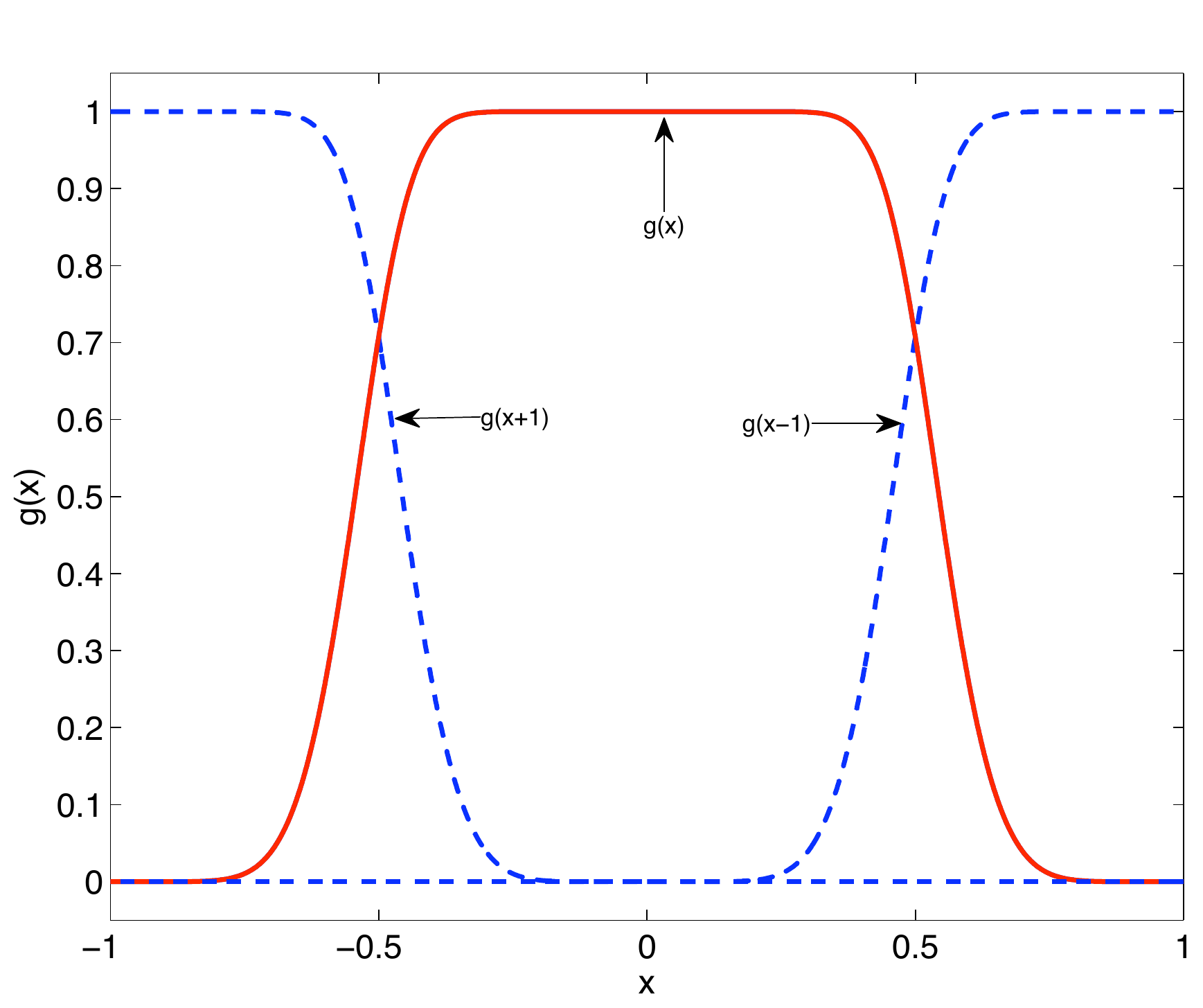}
     \end{center}
\caption{The figure shows  window function $g(x)$ has support $[-1, 1]$,
the left and the right half graphs are the translations of the function.}
\label{f:profilewindowfunctionsine01aaa}
\end{figure}

To design $g(x)$, we consider a monotone increasing function $h(x)$ ($C^k$ or $C^{\infty}$) such that
\begin{equation}\label{f:profilefunction01aa}
 h(x)=\left\{
\begin{array}{ll}
                      0& \text{if} \quad x \le -1,\\
                      1 &  \text{if} \quad x \ge 0,
               \end{array}
\right.
\end{equation}
and satisfies 
\begin{equation} \label{eq:sum_hsquared}
 h^2(-1/2+y)+ h^2\left(-1/2-y \right)=1, \qquad -1 \le x \le 0 .
\end{equation}
Then to form a function $g$ with support on $[-v_0,v_0]$
we take
\[
  g = \left\{ \begin{array}{ll}
                h(x/v_0)  & \text{for $x <0$}\\
                h(-x/v_0) & \text{for $x \ge 0$.}
              \end{array} \right.
\]
then (\ref{eq:sum_hsquared}) ensures that (\ref{findbound0a0:f}) holds.
Of course a variation is possible. For example we can take a parameter
$1/2 < a < 1$, and, let $g(x)=0$ for $|x| > a$, $g(x)=1$ for $|x| < 1-a$,
and $g$ described by a dilated version of $h$ for $1-a < |x| < a$.

In Figure~\ref{f:profilewindowfunctionsine01}, we show some examples of
monotone increasing functions. First we have a class of functions,
with different $c$, given by
\[
h_1(x)  = \left\{
        \begin{array}{ll}
                      0& \text{if} \quad x \le 0,\\
                      \sin\left(\frac{\pi}{2}\frac{e^{-\frac{c}{x}}}{e^{-\frac{c}{x}}+e^{-\frac{c}{1-x}}} \right)  & \text{if}\quad  0 < x < 1,\\
                      1 &  \text{if} \quad x \ge 1,
               \end{array}
\right.
\]
(we need to translate this backward by 1 to have a function $h$ satisfying
the description above). This is constructed
using the standard profile $h_a(x) =
\frac{e^{-\frac{1}{x}}}{e^{-\frac{1}{x}}+e^{-\frac{1}{1-x}}}$, that goes smoothly
from 0 to 1 on $[0,1]$ and satisfies $h_a(x) + h_a(1-x) =1$. Taking the
$\sin(\frac{\pi}{2} h_a(x))$ instead of $h_a$ transforms the property
$h_a(x) + h_a(1-x) =1$ into the similar property (\ref{eq:sum_hsquared})
for the squares.

Two functions from \cite{Mallat2009} are
\[
  h_2(x)  = \left\{
               \begin{array}{ll}
                      0& \text{if}\quad  x \le -1,\\
                      \sin\left(\frac{\pi}{2} \sin^2\left(\frac{\pi}{2} (1+x)\right)\right)  & \text{if} \quad -1 < x < 0,\\
                      1 &  \text{if} \quad x \ge 0.
               \end{array}
\right.
\]
or
\[
h_3(x)=
\left\{
               \begin{array}{ll}
                      0& \text{if}\quad  x \le -1,\\
                      \cos\left(\frac{\pi}{2} \sin^2\left(\frac{\pi}{2} x\right)\right)  & \text{if} \quad -1 < x < 0,\\
                      1 &  \text{if} \quad x \ge 0.
               \end{array}
\right.
\]
We also find another profile function following \cite{Mallat2009}. Here we consider a monotone increasing function (denote it as mallat on the legend of the figures) in $[-1, 0]$
\[
h_4(x)=\left\{
  \begin{array}{ll}
   0 & \text{if} \quad x\le -1\\
\sin\left(\frac{\pi}{4}\left(1+\sin\left(\frac{\pi}{2} \sin\left(\frac{\pi}{2} \sin\left(\frac{\pi (2x+1)}{2}\right)\right)\right)\right)\right) &  \text{if} \quad -1 < x< 0\\
1 & \text{if} \quad x\ge 0.
\end{array}
\right.
\]
It is to note that if we consider  two sine functions or four sine functions inside $\sin(\pi/4(1+f(x))$, then the decay in the frequency domain is not as fast as of considering three sine functions defined by $h_4(x)$, so we decide to use $h_4(x)$ as a representative for this class of profile functions.
%
\begin{figure}[t]
     \begin{center}
     \includegraphics[width=0.9\textwidth,height=8.5cm]{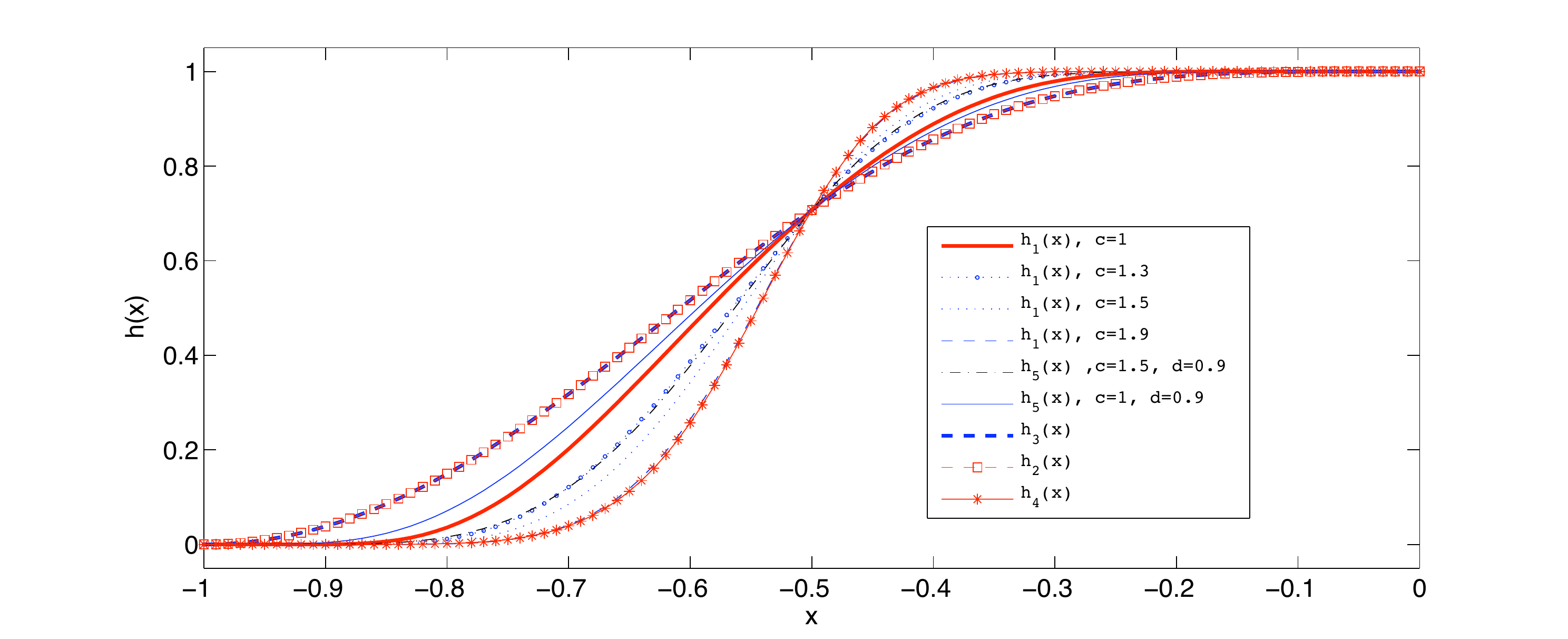}
     \end{center}
     \caption{The figure shows  several monotone increasing  functions $h(x)$ in $[-1, 0]$. 
It is, in fact, showing left half of the window functions considered in this study.
}
     \label{f:profilewindowfunctionsine01}
    \end{figure}

Here to find a good window function which has fast decay in
frequency domain, we try several window functions $g(x)$, where
$g(x)$ is defined from $g$ as given above.
It can be observed that the windows with profile function $h_1(x)$,
and the window with profile function $h_4(x)$ are  smoother than
the other windows and have faster decay all over the frequency domain.

We also define a dilated variant of the profile
$h_1(x)$, effectively throwing out some of the zeros and ones
on the outside of $[0,1]$.
We define the dilated profile function by
\[
  h_5(x)= h_1(1/2+d(x-1/2)),
\qquad \text{for some suitable }\quad 0< d < 1.
\]
\begin{figure}[t]
    \begin{center}
     \includegraphics[width=0.98\textwidth,height=9.5cm]{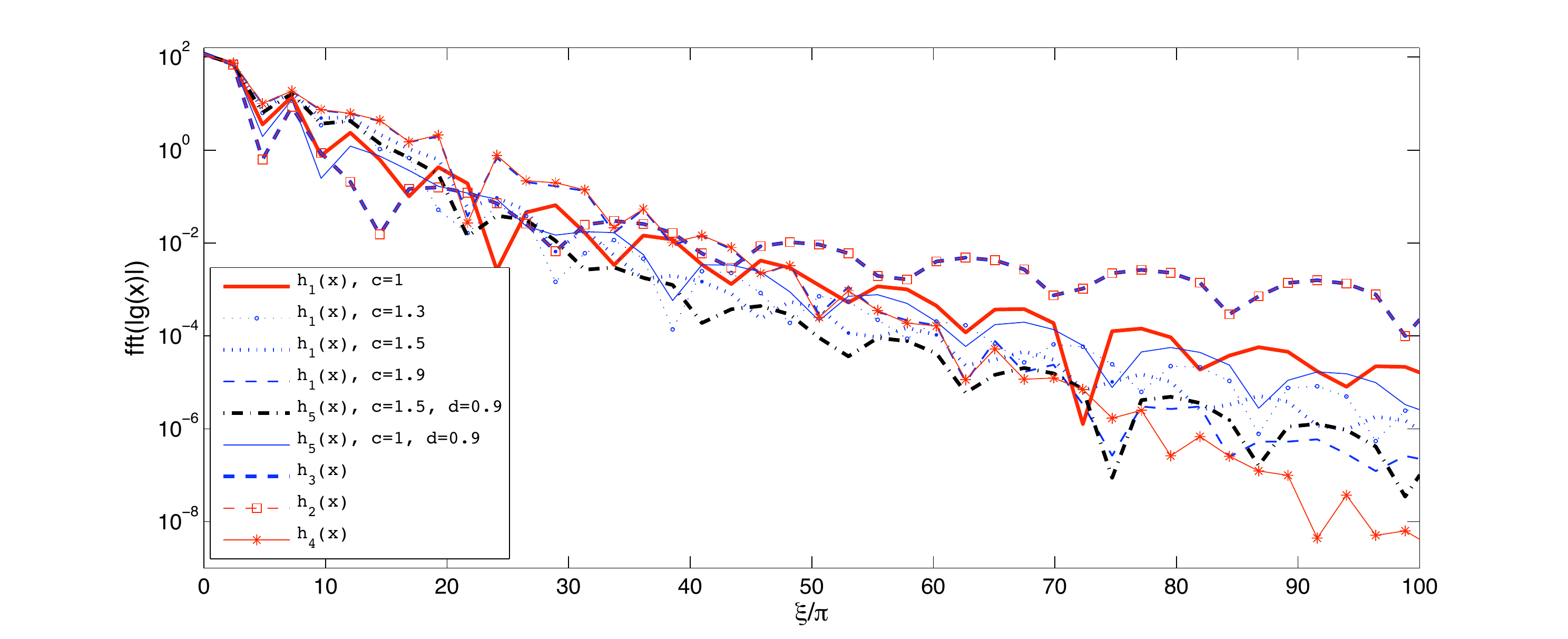}
     \end{center}
\caption{
The figure shows  several $|\hat g|$ in a logarithmic scale.  We consider the window function
with profile function $h_1(x)$
with $c=[1, 1.3, 1.5, 1.7, 1.9]$,
$h_5(x)$ with  $c=1, 1.5$,
and  $d=0.9$,
as well as
considering $h_2(x)$, $h_3(x)$ and $h_4(x)$.
%
}
\label{f:com_alllogplots0a}
 \end{figure}

In Figure~\ref{f:com_alllogplots0a}, we compare several window functions
to find out which one has faster decay in the frequency domain.
Analyzing the graphs, we notice that
the window functions defined from monotone increasing functions  $h_2(x)$ and $h_3(x)$ have slow decay in the frequency domain. All other windows have quite similar decay in the frequency domain.
We observe that the window function formed by the monotone increasing function $h_4(x)$ and the stretched variant with $d=0.9$ and $c=1.5$ behave very
well when $\xi/\pi$ is in the range of 20 to 60, a region that one might
still expect to be relevant, with relative
size of the Fourier coefficients between $10^{-4}$ to $10^{-8}$.

\end{document}